\documentclass{amsart}

\usepackage{amsmath}
\usepackage{amsfonts}
\usepackage{amssymb}
\usepackage{graphicx}
\usepackage{hyperref}
\usepackage{mathrsfs}
\usepackage{pb-diagram}
\usepackage{epstopdf}
\usepackage{amsmath,amsfonts,amsthm,enumerate,amscd,latexsym,curves}
\usepackage{bbm}
\usepackage[mathscr]{eucal}
\usepackage{epsfig,epsf,xypic,epic}
\usepackage{caption}
\usepackage{subcaption}
\newtheorem{theorem}{Theorem}[section]

\newtheorem{lemma}[theorem]{Lemma}
\newtheorem{corollary}[theorem]{Corollary}

\newtheorem{definition}[theorem]{Definition}

\newtheorem{proposition}[theorem]{Proposition}

\newtheorem{remark}[theorem]{Remark}

\newcommand{\dbar}{\bar{\partial}}
\newcommand{\dpsum}[2]{\displaystyle{\sum_{#1}^{#2}}}
\newcommand{\pd}[2]{\frac{\partial #1}{\partial #2}}
\newcommand{\secpd}[3]{\frac{{\partial}^2 #1}{{\partial #2}{\partial #3}}}
\newcommand{\lie}[2]{\mathcal{L}_{#1}{#2}}

\newcommand{\norm}[1]{\|#1\|}

\def\bC{\mathbb{C}}

\begin{document}

\title[A differential-geometric approach to deformations of pairs]{A differential-geometric approach to deformations of pairs $(X,E)$}
\author[K. Chan]{Kwokwai Chan}
\address{Department of Mathematics\\ The Chinese University of Hong Kong\\ Shatin\\ Hong Kong}
\email{kwchan@math.cuhk.edu.hk}
\author[Y.-H. Suen]{Yat-Hin Suen}
\address{Department of Mathematics\\ The Chinese University of Hong Kong\\ Shatin\\ Hong Kong}
\email{yhsuen@math.cuhk.edu.hk}

\date{\today}

\begin{abstract}
This article gives an exposition of the deformation theory for pairs $(X, E)$, where $X$ is a compact complex manifold and $E$ is a holomorphic vector bundle over $X$, adapting an analytic viewpoint \`{a} la Kodaira-Spencer. By introducing and exploiting an auxiliary differential operator, we derive the Maurer--Cartan equation and differential graded Lie algebra (DGLA) governing the deformation problem, and express them in terms of differential-geometric notions such as the connection and curvature of $E$, obtaining a chain level refinement of the classical results that the tangent space and obstruction space of the moduli problem are respectively given by the first and second cohomology groups of the Atiyah extension of $E$ over $X$. As an application, we give examples where deformations of pairs are unobstructed.
%Our results are then applied to study the jumping of the dimension of the cohomology group $H^1(X,\text{End}(T_X))$, answering a question of Huybrechts \cite{Huybrechts95} in the affirmative in the case of algebraic K3 surfaces.
\end{abstract}

\maketitle

\tableofcontents

\section{Introduction}

The theory of deformations of pairs $(X,E)$, where $X$ is a compact complex manifold and $E$ is a holomorphic vector bundle over $X$, has been studied using both algebraic \cite{Sernesi_book, Martinengo_thesis, Huybrechts-Thomas10, Li08} and analytic \cite{Siu-Trautmann81, Huang95} approaches and is well-understood among experts. In this mostly expository paper, we revisit this problem from a viewpoint \`{a} la Kodaira-Spencer \cite{Kodaira-Spencer58, Kodaira-Spencer60, Morrow-Kodaira_book}, emphasizing the use of differential-geometric notions such as connections and curvatures of $E$ and the induced differential operators.
What we obtain is a chain level refinement of the classical results.

%By spelling out all the details, we hope that our down-to-earth approach will be accessible even to beginners. We also expect that our techniques can be applied to other situations, especially when algebraic methods are not available.

To illustrate our strategy, recall that a family of deformations $\{X_t\}_{t\in\Delta}$ of a compact complex manifold $X$ over a small ball $\Delta$ can be represented by elements $\{\varphi_t\}_{t\in\Delta}\subset\Omega^{0,1}(T_X)$, where $T_X$ is the holomorphic tangent bundle of the complex manifold $X$. While the Dolbeault operator $\dbar_t:\Omega^0_{X_t} \to \Omega^{0,1}_{X_t}$ on $X_t$ is not easy to write down explicitly, one may consider the more convenient operator
$$\dbar + \varphi_t \lrcorner \partial:\Omega^0_X \to \Omega^{0,1}_X.$$
Although $\dbar + \varphi_t \lrcorner \partial$ is {\em not} the same as $\bar{\partial}_t$, their kernels coincide (see Proposition \ref{prop:auxiliary_operator_manifold}), and hence $\bar{\partial} + \varphi_t \lrcorner \partial$ completely determines the local holomorphic functions with respective to the complex structure $J_t$ on $X_t$. In fact, we have $\Omega^{0,1}_{X_t}=(id - \bar{\varphi}_t^*)\Omega^{0,1}_X$ and the commutative diagram
\begin{equation*}
\xymatrix{
\Omega^0_X \ar[rd]_{\dbar+\varphi_t\lrcorner\partial} \ar@{->}[r]^{\dbar_t}& {\Omega^{0,1}_{X_t}} \ar[d]^{\pi_t}
\\ & {\Omega_X^{0,1}}
}
\end{equation*}
where $\pi_t$ is the inverse of the canonical projection $P_t:\Omega^{0,1}_X\subset\Omega^1_X\rightarrow\Omega^{0,1}_{X_t}$ (see the proof of Proposition \ref{prop:isomorphic_cohomology}); in this way we can compute everything in terms of the holomorphic structure on $X$.
%This in fact follows from the the construction $\Omega^{0,1}_{X_t}=(id-\overline{\phi_t})\Omega_X^{0,1}$ and $\pi_t\circ\dbar_t=\dbar + \varphi_t \lrcorner \partial$, where $\pi_t:\Omega_{X_t}^{0,1}\subset\Omega^1_X\rightarrow\Omega_X^{0,1}$ is the natural projection, which is an isomorphism for $|t|$ small.%

The same idea can be applied to deformations of pairs. First of all, given a family of deformations $\{(X_t,E_t)\}_{t\in\Delta}$ of $(X,E)$, we have a family of elements $\{\varphi_t\}_{t\in\Delta}\subset\Omega^{0,1}(T_X)$ since $\{X_t\}$ is in particular a family of deformations of $X$. Using a smooth trivialization, we may further assume that $E_0\cong E$ as smooth complex vector bundles. By choosing a hermitian metric on $E$ and considering the associated Chern connection, we define a differential operator
$$\bar{D}_t:\Omega^{0,q}_X(E)\to\Omega^{0,q+1}_X(E),$$
which satisfies the Leibniz rule and $\bar{D}_t^2 = 0$ (see Section \ref{sec:MC_eqn} for details).
While $\bar{D}_t$ is certainly {\em not} the Dolbeault operator $\dbar_{E_t}$ on the holomorphic bundle $E_t$, its kernel gives precisely the space of holomorphic sections of $E_t$ over $X_t$, and similar to the case when $E = \mathcal{O}_X$ (the trivial line bundle), we have a commutative diagram
\begin{equation*}
\xymatrix{
\Omega^0_X(E) \ar[rd]_{\bar{D}_t} \ar@{->}[r]^{\dbar_{E_t}}& {\Omega^{0,1}_{X_t}(E)} \ar[d]^{\pi_t}
\\ & {\Omega_X^{0,1}(E)}
}
\end{equation*}
relating the operators $\bar{D}_t$ and $\dbar_{E_t}$. Furthermore, $\bar{D}_t$ determines a family of elements
$$A_t := \bar{D}_t - \dbar_E - \varphi_t\lrcorner\nabla \in \Omega^{0,1}(\text{End}(E));$$
conversely, given any family of pairs of elements $A_t\in\Omega^{0,1}(\text{End}(E))$, $\varphi_t\in\Omega^{0,1}(T_X)$, we can set
$$\bar{D}_t := \dbar_E + \varphi_t\lrcorner\nabla + A_t.$$
The upshot is the following Newlander--Nirenberg-type theorem for deformations of pairs:
\begin{theorem}[=Theorem \ref{thm:Dbar_integrability}]
Given $\varphi_t\in\Omega^{0,1}(T_X)$ and $A_t\in\Omega^{0,1}(End(E))$, if the induced differential operator $\bar{D}_t$ defined above satisfies $\bar{D}_t^2=0$, then it defines a holomorphic pair $(X_t,E_t)$ (i.e. an integrable complex structure $J_t$ on $X$ together with a holomorphic bundle structure on $E$ over $(X,J_t)$).
\end{theorem}

Applying this, we derive the {\em Maurer--Cartan equation}:
\begin{theorem}[=Theorem \ref{thm:MC_eqn}]
Given a holomorphic pair $(X,E)$ and a smooth family of elements $\{(A_t,\varphi_t)\}_{t\in\Delta}\subset\Omega^{0,1}(A(E))$. Then $(A_t, \varphi_t)$ defines a holomorphic pair $(X_t, E_t)$
if and only if the Maurer--Cartan equation
\begin{align*}
\dbar_{A(E)}(A_t,\varphi_t) + \frac{1}{2}[(A_t,\varphi_t),(A_t,\varphi_t)] = 0
\end{align*}
is satisfied. Here, $A(E)$ is the {\em Atiyah extension} of $E$ which is equipped with the Dolbeault operator $\dbar_{A(E)}$, and the bracket $[-,-]$ is defined in terms of connections and curvatures on $E$ in Proposition \ref{prop:bracket}.
\end{theorem}

Moreover, the triple $(\Omega^{0,\bullet}(A(E)),\dbar_{A(E)},[-,-])$ forms a \emph{differential graded Lie algebra} (DGLA), which (as expected) is naturally isomorphic to the one obtained by algebraic means \cite{Sernesi_book, Martinengo_thesis} (see Appendix \ref{sec:compare_classical}). At this point, we should mention that the relation between deformation theories and DGLAs was first recognized in \cite{NR_cohomology_deformations} by Nijenhuis and Richardson; later, it was suggested by Goldman and Millson \cite{Glodman_Millson_Kahler, Goldman_Millson_kuranishi_space} and many others that deformation problems should always be controlled by DGLAs and solutions to the associated Maurer--Cartan equations form moduli spaces of the deformation problems.

From the Maurer--Cartan equation, we deduce that the space of first order deformations of $(X,E)$ is given by the first cohomology group
$H^{0,1}_{\dbar_{A(E)}} \cong H^1(X, A(E))$
(see Section \ref{sec:1st_order_def}), and that the obstruction theory is captured by the {\em Kuranishi map}
\begin{align*}
Ob_{(X,E)}:U\subset H^1(X, A(E)) & \rightarrow H^2(X, A(E)),\\
\dpsum{i=1}{m}t^i(A_i,\varphi_i) & \mapsto H[(A_t,\varphi_t),(A_t,\varphi_t)],
\end{align*}
whence obstructions lie inside the second cohomology group
$H^{0,2}_{\dbar_{A(E)}} \cong H^2(X, A(E))$
(see Section \ref{sec:obs_kuranishi}). Here, $U$ is a small open neighborhood of the origin $0\in H^1(X,A(E))$. We also give a proof of the existence of a locally complete (or versal) family (see Theorem \ref{thm:completeness}; cf. \cite{Siu-Trautmann81}) using an analytic method originally due to Kuranishi \cite{Kuranishi65}.

Next we apply this analytic approach to look for situations where deformations of holomorphic pairs are unobstructed (Section \ref{sec:unobstr}). The main tool is the following proposition relating deformations of the pair $(X,E)$ to that of $X$ and $E$, which first appeared in \cite[Appendix A]{Huybrechts95} without proof:
\begin{proposition}[=Proposition \ref{prop:obstr_commute}]
Denote the Kuranishi obstruction maps of the deformation theory of $X$, $E$ and $(X,E)$ by $Ob_X$, $Ob_E$ and $Ob_{(X,E)}$ respectively. Then we have the following commutative diagram:
\begin{equation*}
\xymatrix{
\cdots\ar@{->}[r]& H^1(X,\text{End}(E)) \ar[d]_{Ob_E} \ar@{->}[r]^{\iota^*}& {H^1(X, A(E))} \ar[d]_{Ob_{(X,E)}} \ar@{->}[r]^{\pi^*}& H^1(X,T_X) \ar[d]_{Ob_X} \ar@{->}[r]^{\delta} & \cdots
\\\cdots\ar@{->}[r] & {H^2(X,\text{End}(E))} \ar@{->}[r]^{\iota^*} & {H^2(X, A(E))} \ar@{->}[r]^{\pi^*} & H^2(X,T_X) \ar@{->}[r]^{\delta} & \cdots
}
\end{equation*}
Here, the connecting homomorphism $\delta$ is given by contracting with the {\em Atiyah class}:
$$\delta(\varphi)=\varphi\lrcorner [F_{\nabla}].$$
\end{proposition}

\begin{remark}
The vertical maps (Kuranishi maps) in the above proposition are understood to be defined on small neighborhoods around the origins of the corresponding cohomology groups.
\end{remark}

Applying this proposition, we obtain results which generalize some of those in the recent work of X. Pan \cite{Pan13} (where only the case when $E$ is a line bundle was considered). We also prove that when $X$ is a K3 surface and $E$ is a good bundle over $X$ with $c_1(E) \neq 0$ (Proposition \ref{prop:unobstr_K3}), deformations of pairs $(X,E)$ are unobstructed.

%Finally, we investigate a question raised by Huybrechts in \cite{Huybrechts95}, namely, whether the dimension of the cohomology group $H^1(X,\text{End}(T_{X}))$ is invariant under small deformations of projective Calabi-Yau manifolds. By applying the Chern-Weil approach to the deformation theory of pairs, we are able to solve this question in the affirmative in the case of algebraic K3 surfaces.

%\begin{theorem}[=Theorem \ref{theorem:jumping_K3}]
%Suppose that $X$ is an algebraic K3 surface. Then the dimension of the cohomology group $H^1(X_t,\text{End}(T_{X_t}))$ is invariant under any small algebraic deformation of $X$.
%\end{theorem}

\begin{remark}
After we posted an earlier version of this article on the arXiv, Carl Tipler informed us that the paper \cite{Huang95} of L. Huang already contained most of our results, although we have more detailed expositions of first order deformations (Section \ref{sec:1st_order_def}) and the proof of existence of Kuranishi families (Section \ref{sec:completeness}) than Huang's paper and we have a comparison with the algebraic approach (Appendix \ref{sec:compare_classical}) showing in particular that the isomorphism class of the DGLA is independent of the choice of hermitian metric on $E$.
As a result, this article should be regarded as largely expository.
%For these reasons and also for making our paper more self-contained, we retain these sections in this new version.
\end{remark}

\section*{Acknowledgment}
The authors are grateful to Conan Leung, Si Li and Yi Zhang for various illuminating and useful discussions. Thanks are also due to Carl Tipler for pointing out the paper \cite{Huang95} and to Richard Thomas for interesting comments and suggestions on an earlier draft of this article. We would also like to express our gratitude to the referees for carefully reading our manuscript and giving various useful comments and suggestions which helped to greatly improve the exposition of this article.

The work of the first named author described in this paper was substantially supported by grants from the Research Grants Council of the Hong Kong Special Administrative Region, China (Project No. CUHK404412 $\&$ CUHK400213).

\section{Connections, curvature and the Atiyah class}\label{sec:prelim}

In this section, we review some basic notions in the theory of holomorphic vector bundles over complex manifolds and fix our notations. Excellent references for these materials include the textbooks \cite{Griffiths-Harris_book, Huybrechts_book}.

Let $E$ be a complex vector bundle over a smooth manifold $X$. For $k \geq 0$, we denote by $\Omega^k$ the sheaf of $k$-forms and by $\Omega^k(E)$ the sheaf of $E$-valued $k$-forms over $X$.
Recall that a {\em connection} on $E$ is a $\bC$-linear sheaf homomorphism $\nabla: \Omega^0(E) \to \Omega^1(E)$ satisfying the Leibniz rule:
$$\nabla(f\cdot s) = df \otimes s + f \cdot \nabla s$$
for $f \in \Omega^0$ and $s\in \Omega^0(E)$.
We extend $\nabla$ naturally to $\nabla: \Omega^k(E) \to \Omega^{k+1}(E)$ by defining
$$\nabla(\alpha \otimes s) = d\alpha \otimes s + (-1)^k \alpha \wedge \nabla s$$
for $\alpha \in \Omega^k$ and any $s \in \Omega^0(E)$.
The {\em curvature}
$$F_\nabla = \nabla \circ \nabla: \Omega^0(E) \to \Omega^2(E)$$
of $\nabla$ can then be regarded as a global $\text{End}(E)$-valued $2$-form.
Also, $\nabla$ induces a natural connection on $\text{End}(E)$ by
$$(\nabla A)(s)=\nabla(As)-A(\nabla s),$$
where $A \in \Omega^0(\text{End}(E))$ and $s \in \Omega^0(E)$,
and we have the {\em Bianchi identity}
$$\nabla F_{\nabla}=0.$$

Now suppose that $X$ is a complex manifold. For $p,q \geq 0$, we denote by $\Omega^{p,q}$ the sheaf of $(p,q)$-forms and by $\Omega^{p,q}(E)$ the sheaf of $E$-valued $(p,q)$-forms over $X$. Recall that a holomorphic structure on a complex vector bundle $E$ over $X$ is uniquely determined by a $\bC$-linear operator $\dbar_E: \Omega^0(E) \to \Omega^{0,1}(E)$ satisfying the Leibniz rule and the integrability condition $\dbar_E^2 = 0$.
If we further equip $E$ with a hermitian metric $h$, then there exists a unique connection $\nabla$ on $E$ which is hermitian (i.e. $d h(s_1,s_2) = h(\nabla s_1, s_2) + h(s_1, \nabla s_2)$ for any $s_1, s_2 \in \Omega^0(E)$) and compatible with the holomorphic structure on $E$ (i.e. $\nabla^{0,1} = \dbar_E$, where $\nabla^{0,1} = \Pi^{0,1} \circ \nabla$ and $\Pi^{p,q}: \Omega^{p+q}(E) \to \Omega^{p,q}(E)$ is the natural projection map). $\nabla$ is usually called the {\em Chern connection} on $(E, h)$. The curvature $F_\nabla$ of the Chern connection on $(E,h)$ is real and of type $(1,1)$, so the Bianichi identity implies that $\dbar_{\text{End}(E)}F_{\nabla}=0$, and thus this defines a class
$$[F_{\nabla}]\in H^{1,1}(X, \text{End}(E)),$$
called the {\em Atiyah class} of $E$ \cite{Atiyah57}. We have the following lemma.

\begin{lemma}[\cite{Huybrechts_book}, Proposition 4.3.10]
The Atiyah class is independent of the choice of the Hermitian metric.
\end{lemma}

Using the Atiyah class, one can define an extension of $\text{End}(E)$ by $T_X$; indeed, in the language of algebraic geometry, we can interpret the Atiyah class as an element in the extension group $\text{Ext}^1(E\otimes T_X, E) = \text{Ext}^1(T_X, \text{End}(E))$. Consider the smooth vector bundle $A(E):=\text{End}(E)\oplus T_X$ and the differential operator $\bar{\partial}_{A(E)_B}:\Omega^{0}(A(E))\rightarrow\Omega^{0,1}(A(E))$ on $A(E)$ defined by
$$\bar{\partial}_{A(E)_B}:=
\begin{pmatrix}
\bar{\partial}_{\text{End}(E)} & B
\\O & \bar{\partial}_{T_X}
\end{pmatrix},$$
where $B\in\Omega^{0,1}(\text{Hom}(T_X,\text{End}(E)))$ acts on $\Omega^0(T_X)$ by
$$B\wedge\varphi := -(-1)^{|\varphi|}\varphi\lrcorner F_{\nabla}.$$
To simplify notations, from this point on, we will denote the vector bundles $\text{End}(E)$ and $\text{Hom}(T_X,\text{End}(E))$ by $Q$ and $H$ respectively unless specified otherwise.

\begin{proposition}\label{prop:B_closed}
$B\in\Omega^{0,1}(H)$ is $\bar{\partial}_H$-closed.
\end{proposition}
\begin{proof}
This follows from the Bianchi identity $\bar{\partial}_QF_{\nabla}=0$: For any $v \in T_X$,
\begin{align*}
(\dbar_HB)(v)
= \dbar_Q(Bv)+B(\dbar_{T_X}v)
= -\dbar_Q(v\lrcorner F_{\nabla}) + \dbar_{T_X}v\lrcorner F_{\nabla}
= v\lrcorner\dbar_QF_{\nabla}=0.
\end{align*}
\end{proof}

\begin{proposition}
$\left( A(E),\bar{\partial}_{ A(E)_B}\right)$ defines a holomorphic vector bundle over $X$ whose holomorphic structure depends only on the class $[B]$.
\end{proposition}
\begin{proof}
Clearly $\bar{\partial}_{ A(E)_B}$ satisfies the Leibniz rule, so it suffices to prove that $\bar{\partial}_{ A(E)_B}^2 = 0$. But
$\bar{\partial}_{ A(E)_B}^2=0$ if and only if $\bar{\partial}_HB=0$ which holds by Proposition \ref{prop:B_closed}. This proves the first part of the proposition.

To see the second part, suppose that $B'-B=\bar{\partial}_H f$ for some $f\in \text{Hom}(T_X,Q)$. Define the smooth bundle isomorphism $F: A(E)\rightarrow A(E)$ by
$$F:(A,v)\longmapsto(A-fv,v),$$
and extend to $ A(E)$-valued $p$-forms. We compute
\begin{align*}
\bar{\partial}_{ A(E)_{B'}}F(A,v)
& = (\bar{\partial}_Q(A-fv) + B'v, \bar{\partial}_{T_X}v) = (\bar{\partial}_QA - \bar{\partial}_Qfv + Bv + \bar{\partial}_Hfv, \bar{\partial}_{T_X}v)\\
& = (\bar{\partial}_QA + Bv - f\bar{\partial}_{T_X}v, \bar{\partial}_{T_X}v) = F\bar{\partial}_{ A(E)_B}(A,v).
\end{align*}
Hence $F$ in fact defines a holomorphic bundle isomorphism between $( A(E),\bar{\partial}_{ A(E)_B})$ and $( A(E),\bar{\partial}_{ A(E)_{B'}})$. Since the curvature $F_{\nabla}$ differs by an exact $\text{End}(E)$-valued 1-form if another metric was used, this shows that the holomorphic structure of $ A(E)_B$ only depends on the class $[B]$ but not the metric.
\end{proof}

\begin{remark}
Under the Dolbeault isomorphism
$$H^1(X, \text{Hom}(T_X,Q))\cong H^{1,1}(X, Q),$$
the class $[B]$ corresponds to the Atiyah class $[F_{\nabla}]$. Hence the holomorphic structure of $ A(E)$ depends only on the Atiyah class of $E$.
\end{remark}

By abuse of notations, we will now write $\bar{\partial}_{A(E)_B}$ simply as $\bar{\partial}_{ A(E)}$, keeping in mind that a hermitian metric on $E$ has been chosen.
\begin{definition}
The holomorphic vector bundle $\left(A(E),\dbar_{A(E)}\right)$, which is an extension of $Q = \text{End}(E)$ by $T_X$, is called the {\em Atiyah extension} of $E$.
\end{definition}

\section{Maurer--Cartan equations}\label{sec:MC_eqn}

In this section, we start our study of the deformation theory of pairs $(X,E)$. Our goal is to derive the DGLA and Maurer--Cartan equation which govern this deformation problem.

\subsection{Deformations of complex structures and holomorphic vector bundles}

We begin by a brief review of the classical theory of deformations of complex structures and holomorphic vector bundles; the textbooks \cite{Morrow-Kodaira_book} and \cite{Kobayashi_book} are classic references for these theories respectively.

We first recall that a family of deformations $\pi:\mathcal{X} \to \Delta$ of a compact complex manifold $X$ can be represented by a family of sections $\varphi_t \in \Omega^{0,1}(T_X)$, where $T_X$ is the holomorphic tangent bundle of $X$ (or the $i$-eigenbundle of the almost complex structure defining $X$), satisfying the Maurer--Cartan equation
\begin{equation}\label{eqn:MC_manifold}
\dbar_{T_X}\varphi_t+\frac{1}{2}[\varphi_t,\varphi_t]=0.
\end{equation}
An essential ingredient in the proof is the Newlander--Nirenberg Theorem \cite{Newlander-Nirenberg57} which states that any integrable almost complex structure comes from a complex structure.

\begin{proposition}\label{prop:auxiliary_operator_manifold}
Define an operator $\bar{\partial} + \varphi_t \lrcorner \partial: \Omega^0 \to \Omega^{0,1}$ by $f \mapsto \bar{\partial}f + \varphi_t \lrcorner (\partial f)$, where $\lrcorner$ denotes the contraction or interior product.
Then a local smooth function $f$ is holomorphic on $X_t$ if and only if $\left(\bar{\partial} + \varphi_t \lrcorner \partial\right)f = 0$, i.e.
$$\bar{\partial}_t f = 0 \Longleftrightarrow \left(\bar{\partial} + \varphi_t \lrcorner \partial \right)f = 0,$$
where $\bar{\partial}_t$ is the $\bar{\partial}$-operator of the complex manifold $X_t$.
\end{proposition}
\begin{proof}
Let $z^1, \ldots, z^n$ be local holomorphic coordinates on $X$ (where $n$ is the complex dimension of $X$). Then $\varphi_t$ is of the form
\footnote{The Einstein summation convention will be used throughout this article.}
$$\varphi_t =\varphi^j_i(z,t) d\bar{z}^i \otimes \frac{\partial}{\partial z^j}.$$
Hence $T_{X_t}^{0,1}$ is locally spanned by
$$\frac{\partial}{\partial\bar{z}^i} + \varphi^j_i(z,t) \frac{\partial}{\partial z^j}.$$
The result follows.
\end{proof}

Next we recall the deformation theory of holomorphic vector bundles. Let $E\rightarrow X$ be a complex vector bundle over a complex manifold $X$. It is a standard fact in complex geometry that $E$ admits a holomorphic structure if and only if there exists a linear operator $\dbar_E:\Omega^{0,q}(E)\rightarrow\Omega^{0,q+1}(E)$ satisfying $\dbar_E^2=0$ and the Leibniz rule
$$\dbar_E(\alpha\otimes s)=\dbar\alpha\otimes s+(-1)^{|\alpha|}\alpha\wedge\dbar_E(s),$$
for any $\alpha\in\Omega^{0,q}(E)$ and smooth section $s$ of $E$ (we call this the linearized version of the Newlander--Nirenberg Theorem; see e.g. \cite[Theorem 2.6.26]{Huybrechts_book} or \cite[Theorem 3.2]{Moroianu_book}).
Hence if we have a family of holomorphic vector bundles $\mathcal{E}\rightarrow\Delta$ (or $\{E_t\}_{t\in \Delta}$) on $X$, then we have a family of Dolbeault operators $\dbar_{E_t}$, whose squares are zero and all satisfy the Leibniz rule.
\begin{proposition}
Given a family of deformations $\{E_t\}_{t \in \Delta}$ of $E$, the element
$A_t:=\dbar_{E_t}-\dbar_E\in\Omega^{0,1}(\text{End}(E))$ satisfies the Maurer--Cartan equation
$$\dbar_{\text{End}(E)}A_t+\frac{1}{2}[A_t,A_t]=0$$
for all $t\in \Delta$.
Conversely, if we are given a family $\{A_t\}_{t\in\Delta}\subset \Omega^{0,1}(\text{End}(E))$ which satisfies the Maurer--Cartan equation for each $t$, then $\left\{(E,\dbar_E+A_t)\right\}_{t\in\Delta}$ defines a family of deformations of $E$.
\end{proposition}
\begin{proof}
Note that
$$(\dbar_E + A_t)^2 = \dbar_EA_t + A_t\dbar_E + A_t\wedge A_t = \dbar_{\text{End}(E)}A_t + \frac{1}{2}[A_t,A_t].$$
The result follows from the linearized version of the Newlander--Nirenberg Theorem.
\end{proof}

\subsection{Deformations of holomorphic pairs and the operator $\bar{D}_t$}

\begin{definition}
A {\em holomorphic pair} $(X,E)$ consists of a compact complex manifold $X$ together with a holomorphic vector bundle $E$ over $X$.
\end{definition}

\begin{definition}
Let $(X,E)$ be a holomorphic pair.
A {\em family of deformations} of $(X,E)$ over a small ball $\Delta$ centered at the origin in $\bC^d$ consists of a proper and submersive holomorphic map $\pi: \mathcal{X} \to \Delta$ (a family of deformations of $X$ over $\Delta$) and a holomorphic vector bundle $\mathcal{E} \to \mathcal{X}$ such that $\pi^{-1}(0) = X$ and $\mathcal{E}|_{\pi^{-1}(0)} = E$. For $t \in \Delta$, we denote by $(X_t, E_t)$ the holomorphic pair parametrized by $t$.
\end{definition}

By the theorem of Ehresmann, if $\Delta$ is chosen to be small enough, the family $\mathcal{X}$ is smoothly trivial, i.e. one can find a diffeomorphism $F:\mathcal{X}\rightarrow \Delta\times X$. Restricting to a fiber $\mathcal{X}_t\subset\mathcal{X}$, one can push forward the complex structure on $\mathcal{X}_t$ to define $J_t$ on $X_t:=\{t\}\times X$ via $F$. One can also trivialize $\mathcal{E}$ as $\Delta\times E$ by a smooth bundle isomorphism $P$ and the holomorphic structure on $E_t:=\{t\}\times E$ is induced from that on $\mathcal{E}|_{\mathcal{X}_t}$ via the map $P$. Hence we can assume that our family is a smoothly trivial family $\Delta\times E\rightarrow\Delta\times X$ over a small ball $\Delta$ in $\mathbb{C}^d$ centered at the origin.

Now let $\{(X_t,E_t)\}_{t \in \Delta}$ be a family of deformations of $(X,E)$. By definition, $\{X_t\}_{t \in \Delta}$ is a family of deformations of $X$, so it can be represented by an analytic family of sections $\varphi_t \in \Omega^{0,1}\left(T_X\right)$ satisfying the Maurer--Cartan equation \eqref{eqn:MC_manifold}.
Define the operator $\bar{D}_t:\Omega^{0,q}(E) \to \Omega^{0,q+1}(E)$ by
\begin{align*}
\bar{D}_t(s^ke_k) = (\bar{\partial}+\varphi_t\lrcorner\partial)s^k\otimes e_k,
\end{align*}
where $\{e_k\}$ is a local holomorphic frame of $E_t$.
\begin{proposition}
The linear operator $\bar{D}_t$ is well-defined, that is, independent of the local holomorphic frame of $E_t$. Moreover, it satisfies the Leibniz rule
$$\bar{D}_t(\alpha\otimes s)=(\bar{\partial}+\varphi_t\lrcorner\partial)\alpha\otimes s+(-1)^{|\alpha|}\alpha\wedge\bar{D}_t(s)$$
for any $s\in\Omega^0(E)$ and $\alpha\in\Omega^{0,*}(X)$. Also, $\bar{D}_t(s)=0$ if and only if $\bar{\partial}_{E_t}(s)=0$.
\end{proposition}
\begin{proof}
To prove well-definedness, we need to show that $\bar{D}_t$ is independent of the choice of a local holomorphic frame $\{e_k\}$ of $E_t$. So suppose $\{f_j\}$ is another local holomorphic frame of $E_t$. Let $\tau^k_j$ be local holomorphic functions on $X_t$ such that $f_j=\tau^k_je_k$. Then for a local section $s=s^ke_k=\widetilde{s}^jf_j$, we have $\widetilde{s}^j=s^k\tau^j_k$ and thus
\begin{align*}
\bar{D}_t(\widetilde{s}^jf_j)
& = (\dbar+\varphi_t\lrcorner\partial)\widetilde{s}^j\otimes f_j = (\dbar+\varphi_t\lrcorner\partial)(s^k\tau^j_k)\otimes f_j\\
& = (\dbar+\varphi_t\lrcorner\partial)s^k\otimes\tau^j_kf_j = \bar{D}_t(s^ke_k).
\end{align*}
Hence $\bar{D}_t$ is well-defined.

The Leibniz rule for $\bar{D}_t$ is clear since $\dbar$ and $\partial$ both satisfy the usual Leibniz rule. Finally, for a smooth section $s$ of $E$, if we write $s=s^ke_k$ locally with $\{e_k\}$ a local holomorphic frame of $E_t$, then we have
$$\bar{D}_t(s)=0 \Longleftrightarrow (\dbar+\varphi_t\lrcorner\partial)s^k=0 \Longleftrightarrow \dbar_ts^k=0 \Longleftrightarrow \dbar_{E_t}(s)=0.$$
\end{proof}

We claim that $\bar{D}_t^2=0$. By our definition of $\bar{D}_t$, for any smooth function $f:X\rightarrow\mathbb{C}$ and local nowhere vanishing holomorphic section $e$ of $E_t$, we have
$$\bar{D}_t^2(fe)=(\bar{\partial}+\varphi_t\lrcorner\partial)^2f\otimes e.$$
To compute the right hand side, we need the following
\begin{lemma}\label{lem:dbar_contract}
For any $\varphi\in\Omega^{0,p}(T_X)$ and $\alpha\in\Omega^1(E)$, we have the Leibniz rule
$$\bar{\partial}_E(\varphi\lrcorner\alpha)=\bar{\partial}_{T_X}\varphi\lrcorner\alpha-(-1)^p\varphi\lrcorner\bar{\partial}_E\alpha.$$
\end{lemma}
\begin{proof}
Writing $\varphi=\varphi_J^i d\bar{z}^J \otimes \frac{\partial}{\partial z^i}$, we have
$$\varphi\lrcorner\alpha=\varphi_J^i d\bar{z}^J\otimes\alpha\left(\pd{}{z^i}\right).$$
Let $\alpha_i:=\alpha(\pd{}{z^i})\in\Omega^0(E)$. Then
\begin{align*}
\bar{\partial}_E(\varphi\lrcorner\alpha)
& = (\bar{\partial}\varphi_J^i\wedge d\bar{z}^J)\otimes\alpha_i+(-1)^p\varphi_J^i d\bar{z}^J\wedge\bar{\partial}_E\alpha_i\\
& = \bar{\partial}_{T_X}\varphi\lrcorner\alpha+(-1)^p\varphi_J^i d\bar{z}^J\wedge\bar{\partial}_E\alpha_i.
\end{align*}
To compute the last term, first note that the contraction of $\frac{\partial}{\partial z^i}$ with $\alpha$ is taken in the $(1,0)$-part, we can therefore assume $\alpha=\alpha_i^kdz^i\otimes e_k$, where $\{e_k\}$ is a local holomorphic frame of $E$. So we have
\begin{align*}
\varphi_J^i d\bar{z}^J\wedge\bar{\partial}_E\alpha_i
& = -\varphi_J^i d\bar{z}^J\wedge\left(\frac{\partial}{\partial z^i}\lrcorner\bar{\partial}\alpha_l^k\wedge dz^l\right)\otimes e_k\\
& = -\varphi_J^i d\bar{z}^J\otimes\frac{\partial}{\partial z^i}\lrcorner\bar{\partial}_E(\alpha_l^kdz^l\otimes e_k) = -\varphi\lrcorner\bar{\partial}_E\alpha,
\end{align*}
and hence the desired formula.
\end{proof}

We can now compute $\bar{D}_t^2$.
\begin{lemma}\label{lem:operator_square_mfd}
For any smooth function, $f:X\rightarrow\mathbb{C}$, we have the equality
$$(\bar{\partial}+\varphi_t\lrcorner\partial)^2f=\left(\dbar_{T_X}\varphi_t+\frac{1}{2}[\varphi_t,\varphi_t]\right)\lrcorner\partial f.$$
\end{lemma}
\begin{proof}
First, we have
$$(\bar{\partial} + \varphi_t\lrcorner\partial)^2f = \dbar(\varphi_t\lrcorner\partial f) + \varphi_t\lrcorner\partial\dbar f + \varphi_t\lrcorner\partial(\varphi_t\lrcorner\partial f).$$
By Lemma \ref{lem:dbar_contract}, the first term is given by
$$\dbar(\varphi_t\lrcorner\partial f)=\dbar_{T_X}\varphi_t\lrcorner\partial f+\varphi_t\lrcorner\dbar\partial f.$$
Since $\partial\dbar=-\dbar\partial$, we have
$$\dbar(\varphi_t\lrcorner\partial f)+\varphi_t\lrcorner\partial\dbar f=\dbar_{T_X}\varphi_t\lrcorner\partial f.$$
For the last term, by writing $\varphi_t=\varphi^l_md\bar{z}^m\otimes\pd{}{z^l}$ in local coordinates, we have
$$\varphi_t\lrcorner\partial f=\varphi^l_m\pd{f}{z^l}d\bar{z}^m,$$
and so
$$\varphi_t\lrcorner\partial(\varphi_t\lrcorner\partial f)=\varphi^i_j\pd{\varphi^l_m}{z^i}\pd{f}{z^l}d\bar{z}^j\wedge d\bar{z}^m+\varphi^i_j\varphi^l_m\secpd{f}{z^i}{z^l}d\bar{z}^j\wedge d\bar{z}^m$$
But
$$\varphi^i_j\varphi^l_m\secpd{f}{z^i}{z^l}d\bar{z}^j\wedge d\bar{z}^m=-\varphi^l_m\varphi^i_j\secpd{f}{z^l}{z^i}d\bar{z}^m\wedge d\bar{z}^j,$$
so we obtain
$$\varphi_t\lrcorner\partial(\varphi_t\lrcorner\partial f)=\varphi^i_j\pd{\varphi^l_m}{z^i}\pd{f}{z^l}d\bar{z}^j\wedge d\bar{z}^m=\frac{1}{2}[\varphi_t,\varphi_t]\lrcorner\partial f.$$
The result follows.
%$$(\dbar+\varphi_t\lrcorner\partial)^2f=\dbar_{T_X}\varphi_t\lrcorner\partial f+(\varphi_t\lrcorner\partial)(\varphi_t\lrcorner\partial f)=\left(\dbar_{T_X}\varphi_t+\frac{1}{2}[\varphi_t,\varphi_t]\right)\lrcorner\partial f.$$
\end{proof}

As $\{X_t\}_{t\in \Delta}$ is an honest family of deformations of $X$, the Maurer--Cartan equation \eqref{eqn:MC_manifold} for $\varphi_t$ holds. Hence we have
\begin{proposition}
$\bar{D}_t^2=0$.
\end{proposition}

From the viewpoint of Proposition \ref{prop:auxiliary_operator_manifold}, it is natural to compare the operator $\bar{D}_t$ with $\dbar_E+\varphi_t\lrcorner\nabla$.
\begin{proposition}
$A_t:=\bar{D}_t-\bar{\partial}_E-\varphi_t\lrcorner\nabla\in\Omega^{0,1}(\text{End}(E))$.
\end{proposition}
\begin{proof}
Let $f$ be a smooth function and $s$ a smooth section of $E$. Using the Leibniz rules, and the fact that the contraction is only taken in the $(1,0)$-part, we have
\begin{align*}
A_t(fs)
& = (\bar{\partial}+\varphi_t\lrcorner\partial)f\otimes s+f\bar{D}_t(s)-\bar{\partial}f\otimes s-f\bar{\partial}_E(s)-\varphi_t\lrcorner\nabla(fs)\\
& = (\varphi_t\lrcorner\partial)f\otimes s+f\bar{D}_t(s)-f\bar{\partial}_E(s)-\varphi_t\lrcorner df\otimes s-f\varphi_t\lrcorner\nabla(s) = fA_t(s).
\end{align*}
\end{proof}

In the other direction, suppose we are now given elements $A_t\in\Omega^{0,1}(\text{End}(E))$ and $\varphi_t\in\Omega^{0,1}(T_X)$, parameterized by $t\in\Delta$, we can then define an operator $\bar{D}_t:\Omega^{0}(E)\rightarrow\Omega^{0,1}(E)$ by
$$\bar{D}_t:=\dbar_E+\varphi_t\lrcorner\nabla+A_t.$$
We extend $\bar{D}_t$ to $\Omega^{0,q}(E)$ in the obvious way, so that the Leibniz rule
$$\bar{D}_t(\alpha\otimes s)=(\dbar+\varphi_t\lrcorner\partial)\alpha\otimes s+(-1)^{|\alpha|}\alpha\wedge\bar{D}_ts$$
holds.
We want to show that if $\bar{D}_t^2=0$, then $(A_t,\varphi_t)$ defines a holomorphic pair $(X_t,E_t)$. First of all, we have
\begin{proposition}
If $\bar{D}_t^2=0$, then $X_t$ is a complex manifold.
\end{proposition}
\begin{proof}
Using the Leibniz rule, we have for any smooth function and sections of $E$ that
$$0=\bar{D}^2(fs)=(\dbar+\varphi_t\lrcorner\partial)^2f\otimes s.$$
Hence $(\dbar+\varphi_t\lrcorner\partial)^2=0$, which is equivalent to the Maurer--Cartan equation \eqref{eqn:MC_manifold} by Lemma \ref{lem:operator_square_mfd}. Therefore, the almost complex structure defined by $\varphi_t$ is integrable.
\end{proof}

We now need to show that $E$ also admits a holomorphic structure over $X_t$. We will follow the approach of \cite{Moroianu_book}. Let us first make the following assertion:
\begin{center}
\textit{Any smooth sections of E can locally be written as $s^ke_k$,}
\end{center}
where $\{e_k\}\subset \text{ker}(\bar{D}_t)$. We can then define $\bar{\partial}_{E_t}$ by
$$\bar{\partial}_{E_t}(s^ke_k):=\bar{\partial}_ts^k\otimes e_k.$$
To check that it is well-defined, suppose we have another local basis $\{f_j\}\subset \text{ker}(\bar{D}_t)$, then there exist $(h_j^k)$ such that $f_j=h_j^ke_k$. Applying $\bar{D}_t$, we have
$$(\dbar+\varphi_t\lrcorner\partial)h_j^k\otimes e_k=0.$$
Since $\{e_k\}$ is assumed to be a local basis, we have $(\dbar+\varphi_t\lrcorner\partial)h_j^k=0$, which is equivalent to $\dbar_th_j^k=0$. Hence
$$\dbar_{E_t}(\widetilde{s}^jf_j)=\dbar_{t}\widetilde{s}^j\otimes f_j=\dbar_t(s^kh_k^j)\otimes f_j=\dbar_t(s^j)\otimes h_k^jf_j=\dbar_{E_t}(s^ke_k).$$
This proves well-definedness.

Clearly, it satisfies the Leibniz rule
$$\dbar_{E_t}(\alpha\otimes s)=\dbar_t\alpha\otimes s+(-1)^{|\alpha|}\alpha\wedge\dbar_{E_t}s$$
and $\bar{\partial}_{E_t}^2=0$ since $\varphi_t$ defines an integrable complex structure on $X$. Hence by the linearized version of the Newlander--Nirenberg Theorem, $E_t = (E, \dbar_{E_t})$ is a holomorphic vector bundle over $X_t$.

It remains to prove that our assertion is correct:
\begin{lemma}
$\text{ker}(\bar{D}_t)$ generates $\Omega^{0}(E)$ locally.
\end{lemma}
\begin{proof}
Let us first fix a smooth local frame $\{\sigma_k\}$ of $E_t$ over a coordinate neighborhood $U\subset X_t$. What we need are coordinate changes $(f^i_j(z,t))\in \Gamma_{sm}(U,GL_r(\mathbb{C}))$ such that $f^i_j\sigma_i\in \text{ker}(\bar{D}_t)$. Writing $\bar{D}_t\sigma_i=\tau^k_i\otimes\sigma_k$ with $\tau^k_i\in\Omega^{0,1}(X)$, the existence of $(f^i_j(z,t))$ is equivalent to
$$0=\bar{D}_t(f^i_j\sigma_i)=(\bar{\partial}f^i_j+\varphi_t\lrcorner\partial f^i_j)\otimes\sigma_i+f^i_j\tau^k_i\sigma_k.$$
This in turn is equivalent to the following system of PDEs
$$(\bar{\partial}+\varphi_t\lrcorner\partial)f^k_j+f^i_j\tau^k_i=0$$
subject to the condition:
$$\bar{D}_t^2=0\Longleftrightarrow (\bar{\partial}+\varphi_t\lrcorner\partial)\tau^i_j=\tau^i_k\wedge\tau^k_j.$$
We will show that this system is solvable, following the line of proof in \cite[Theorem 9.2]{Moroianu_book} (linearized version of the Newlander--Nirenberg Theorem).

First of all we set
$$N:=U\times\mathbb{C}^r,\ T:=span\{dz^{\alpha}-\varphi_t\lrcorner dz^{\alpha},dw_i-\tau_i^kw_k\}.$$
We want to show that $d(T)\subset\Omega^0(\bigwedge^1_{\mathbb{C}}N)\wedge T$. First we have
\begin{align*}
d(dz^{\alpha}-\varphi_t\lrcorner dz^{\alpha})
& = \frac{\partial\varphi^{\alpha}_{\beta}}{\partial z^{\gamma}}d\bar{z}^\beta\wedge dz^{\gamma} - \bar{\partial}_{T_X}\varphi_t\lrcorner dz^{\alpha}.
\end{align*}
Then applying the Maurer--Cartan equation \eqref{eqn:MC_manifold} gives
\begin{align*}
d(dz^{\alpha}-\varphi_t\lrcorner dz^{\alpha})
& = \frac{\partial\varphi^{\alpha}_{\beta}}{\partial z^{\gamma}}d\bar{z}^\beta\wedge\left(dz^{\gamma} - \varphi^{\mu}_{\eta}d\bar{z}^{\eta}\otimes\frac{\partial}{\partial z^{\mu}}\lrcorner dz^{\gamma}\right)\\
& = \frac{\partial\varphi^{\alpha}_{\beta}}{\partial z^{\gamma}}d\bar{z}^\beta\wedge(dz^{\gamma}-\varphi_t\lrcorner dz^{\gamma}).
\end{align*}

Secondly,
\begin{align*}
d(dw_i-\tau_i^lw_l)
& = -\partial\tau_i^lw_l - \bar{\partial}\tau_i^lw_l + \tau_i^l\wedge dw_l\\
& = -(\partial - \varphi_t\lrcorner\partial)\tau_i^lw_l - (\bar{\partial} + \varphi_t\lrcorner\partial)\tau_i^lw_l + \tau_i^l\wedge dw_l\\
& = -(\partial - \varphi_t\lrcorner\partial)\tau_i^lw_l - \tau_i^k\wedge\tau_k^lw_l + \tau_i^l\wedge dw_l\\
& = -(\partial - \varphi_t\lrcorner\partial)\tau_i^lw_l - \tau_i^k\wedge(dw_k-\tau_k^lw_l).
\end{align*}

Hence by the Newlander--Nirenberg Theorem, we obtain holomorphic coordinates $(\zeta^{\alpha}_t,u_i^t)$ on $N$ and smooth functions $F^{\alpha}_{\beta}(\zeta_t)=F^{\alpha}_{\beta}(z,t)$, $F^l_i(\zeta_t,u_t)=F^l_i(z,u,t)$, and $F^l_{\alpha}(\zeta_t,u_t)=F_{l\alpha}(z,w,t)$ such that
\begin{equation*}
\left\{
\begin{split}
d\zeta^{\alpha}_t =& F^{\alpha}_{\beta}(z,t)(dz^{\beta} - \varphi_t\lrcorner dz^{\beta}),\\
du_l^t =& F_l^i(z,w,t)(dw_i - \tau_i^kw_k) + F_{l\alpha}(z,w,t)(dz^{\alpha} - \varphi_t\lrcorner dz^{\alpha}).
\end{split}
\right.
\end{equation*}
Since $\{dz^{\alpha}-\varphi_t\lrcorner dz^{\alpha},dw_i-\tau_i^kw_k\}$ and $\{d\zeta^{\alpha}_t,du^i_t\}$ are basis of $T$, we see that the $(n+r)\times(n+r)$-matrix
$$\begin{pmatrix}
(F^{\alpha}_{\beta}) & (F_{i\alpha})
\\O_{r\times n} & (F_l^i)
\end{pmatrix}$$
is invertible for all $(z,w,t)$. It follows that $(F_l^i)$ is also invertible for all $(z,w,t)$.

Applying the exterior differential on $N$ and evaluating at $w=0$, we have
$$0=dF_l^i\wedge dw_i+F_l^i\tau_i^k\wedge dw_k+dF_{l\alpha}\wedge dz^{\alpha}-dF_{l\alpha}\wedge\varphi_t\lrcorner dz^{\alpha}-F_{l\alpha}d(\varphi_t\lrcorner dz^{\alpha}).$$
Comparing the $dz\wedge dw$-component on both sides gives $\partial_zF_l^i\wedge dw_i+\partial_wF_{l\alpha}\wedge dz^{\alpha}=0$, which implies, by contracting with $\varphi_t$, that
$$\varphi_t\lrcorner dF_l^i\wedge dw_i+\partial_wF_{l\alpha}\wedge\varphi_t\lrcorner dz^{\alpha}=0.$$
Then by comparing the $d\bar{z}\wedge dw$-component, we have
$$\bar{\partial}F_l^i\wedge\partial_zw_i+F_l^i\tau_i^k\wedge dw_k-\partial_wF_{l\alpha}\wedge\varphi_t\lrcorner dz^{\alpha}=0.$$
Together with the formula we just obtained, we arrive at
$$(\bar{\partial}+\varphi_t\lrcorner\partial)F_l^i(z,0)+F_l^k(z,0)\tau_k^i=0.$$
The result now follows by setting $f_i^j(z,t):=F_i^j(z,0,t)$.
\end{proof}

In summary, we have proved the following
\begin{theorem}\label{thm:Dbar_integrability}
Given $A_t\in\Omega^{0,1}(\text{End}(E))$ and $\varphi_t\in\Omega^{0,1}(T_X)$. If the induced differential operator $\bar{D}_t:\Omega^{0,q}(E)\rightarrow\Omega^{0,q+1}(E)$ satisfies $\bar{D}_t^2=0$ and the Leibniz rule
$$\bar{D}_t(\alpha\otimes s)=(\bar{\partial}+\varphi_t\lrcorner\partial)\alpha\otimes s+(-1)^{|\alpha|}\alpha\wedge\bar{D}_t(s),$$
then $E$ admits a holomorphic structure over the complex manifold $X_t$, which we will denote by $E_t\to X_t$ or just $E_t$.
\end{theorem}

The operator $\bar{D}_t$ gives a cochain complex
$$(\Omega^{0,\bullet}(E),\bar{D}_t).$$
It is then natural to compare the cohomologies $H^{\bullet}(X_t,E_t)$ and $H^{\bullet}(\Omega^{0,\bullet}(E),\bar{D}_t)$. But $\bar{D}_t$ captures only the holomorphicity of the pair $(X_t,E_t)$, so we would not expect $H^{\bullet}(\Omega^{0,\bullet}(E),\bar{D}_t)$ to be something new.

\begin{proposition}\label{prop:isomorphic_cohomology}
For any $t\in \Delta$, we have the isomorphism
$$H^q(X_t,E_t)\cong H^q(\Omega^{0,\bullet}(E),\bar{D}_t)$$
for any $q\geq 0$.
\end{proposition}
\begin{proof}
We first prove the case when $E=\mathcal{O}_X$ and $q=0$. Let $P_t:\Omega^{0,1}_X\rightarrow\Omega^{0,1}_{X_t}$ be the restriction of the projection $\Omega^1_X\rightarrow\Omega^{0,1}_{X_t}$. Since $P_t$ is an isomorphism for $|t|$ small, it suffices to prove that
$$\dbar_tP_t=P_t(\dbar+\varphi_t\lrcorner\partial)$$
at every point $x\in X$. So let us fix $x\in X$ and let $\{z^j\}$ be local complex coordinates around $x$. Let $\bar{v}_j:=\pd{}{\bar{z}^j}+\varphi^k_j\pd{}{z^k}$ and $\bar{\epsilon}^j$ be its dual vector. Then Maurer--Cartan equation of $\varphi_t$ implies
$$[\bar{v}_j,\bar{v}_k]=0.$$
By the Newlander-Nirenberg theorem, we have complex coordinates $\{\zeta^j\}$ on $X_t$ such that
$$\pd{}{\bar{\zeta}^j}=\bar{v}^j\text{ and }d\bar{\zeta}^j=\bar{\epsilon}^j$$
at the point $x$. Then at $x$,
$$P_t(\dbar+\varphi\lrcorner\partial)f=\bigg(\pd{f}{\bar{z}^j}+\varphi_j^k\pd{f}{z^k}\bigg)P_t(d\bar{z}^j).$$
We need to show that $P_t(d\bar{z}^j)=\bar{\epsilon}^j$. We write
$$d\bar{z}^j=c^j_k\bar{\epsilon}^j+d_k^j\epsilon^k.$$
Then
$$c^j_k=d\bar{z}^j(\bar{v}_k)=d\bar{z}^j\bigg(\pd{}{\bar{z}^j}+\varphi_j^k\pd{}{z^k}\bigg)=\delta^i_k.$$
Hence $P(d\bar{z}^j)=\bar{\epsilon}^j$. Therefore,
$$P_t(\dbar+\varphi\lrcorner\partial)f=(\bar{v}_jf)\bar{\epsilon}^j=\pd{f}{\bar{\zeta}^j}d\bar{\zeta}^j=\dbar_tP_tf$$
at $x$. Since $x$ is arbitrary, the case $q=0$ is done.

For $q>0$. By abusing the notation, we still denote the induced projection $\Omega^{0,q}_X\rightarrow\Omega^{0,q}_{X_t}$ by $P_t$. Let $\alpha=\alpha_Jd\bar{z}^J$. Then at the point $x$,
$$P_t(\dbar+\varphi_t\lrcorner\partial)\alpha=P_t(\dbar+\varphi_t\lrcorner\partial)(\alpha_J)\wedge P_t(d\bar{z}^J)=\dbar_tP\alpha_J\wedge\bar{\epsilon}^J.$$
We need to show that $\dbar_t(\bar{\epsilon}^j)=0$ for all $j$. Since $\{\bar{\epsilon}^j\}$ is a local frame of $(T^{0,1}_{X_t})^*$, $d\bar{\epsilon}^j\in\Omega^{1,1}_{X_t}\oplus\Omega^{0,2}_{X_t}$. Hence, in order to prove $\dbar_t(\bar{\epsilon}^j)=0$, it suffices to show that $d\bar{\epsilon}^j(\bar{v}_k,\bar{v}_l)=0$ for all $k,l$. This follows from
$$d\bar{\epsilon}^j(\bar{v}_k,\bar{v}_l)=\bar{v}_k\bar{\epsilon}^j(\bar{v}_l)-\bar{v}_l\bar{\epsilon}^j(\bar{v}_k)-\bar{\epsilon}^j([\bar{v}_k,\bar{v}_l])=0.$$
This completes the case $E=\mathcal{O}_X$ and $q\geq 0$.

For a general holomorphic vector bundle $E$, $\bar{D}_t$ is locally given by
$$\bar{D}_t(\alpha^j\otimes e_j)=(\dbar+\varphi_t\lrcorner\partial)\alpha^j\otimes e_j,$$
where $\{e_j\}$ are local holomorphic frame of $E_t$ over $X_t$. In this case, $P_t$ is extended by
$$\alpha^j\otimes e_j\mapsto P_t(\alpha^j)\otimes e_j.$$
The required relation follows immediately from the $E = \mathcal{O}_X$ case.
\end{proof}

\subsection{DGLA and the Maurer--Cartan equation}

We are now ready to derive the Maurer--Cartan equation governing the deformations of pairs.
Given $A_t\in\Omega^{0,1}(\text{End}(E))$, $\varphi_t\in\Omega^{0,1}(T_X)$ such that the induced differential operator $\bar{D}_t$ satisfies $\bar{D}_t^2=0$, we have
$$(\dbar_E + \varphi_t\lrcorner\nabla+A_t)^2 = \bar{D}_t^2 = 0.$$

Let us expand the left hand side:
\begin{align*}
(\dbar_E + \varphi_t\lrcorner\nabla + A_t)^2
= & \dbar_E(\varphi_t\lrcorner\nabla) + \varphi_t\lrcorner\nabla\dbar_E + \varphi_t\lrcorner\nabla(\varphi_t\lrcorner\nabla)\\
  & \qquad + \dbar_EA_t + A_t\dbar_E + \varphi_t\lrcorner\nabla A_t + A_t(\varphi_t\lrcorner\nabla) + A_t\wedge A_t\\
= & \dbar_E(\varphi_t\lrcorner\nabla) + \varphi_t\lrcorner\nabla\dbar_E + \varphi_t\lrcorner\nabla(\varphi_t\lrcorner\nabla)\\
  & \qquad + \dbar_QA_t + \varphi_t\lrcorner\nabla^QA_t + A_t\wedge A_t.
\end{align*}
Applying Lemma \ref{lem:dbar_contract} to the term $\dbar_E(\varphi_t\lrcorner\nabla)$, we get
\begin{align*}
(\dbar_E + \varphi_t\lrcorner\nabla + A_t)^2
= & \dbar_{T_X}\varphi_t\lrcorner\nabla+\varphi_t\lrcorner(\dbar_E\nabla+\nabla\dbar_E) + \varphi_t\lrcorner\nabla(\varphi_t\lrcorner\nabla)\\
  & \qquad + \dbar_QA_t + \varphi_t\lrcorner\nabla^QA_t + A_t\wedge A_t.
\end{align*}
Since $\nabla$ is the Chern connection, we have $F_{\nabla}=\dbar_E\nabla+\nabla\dbar_E$, and so
$$(\dbar_E + \varphi_t\lrcorner\nabla + A_t)^2 = (\dbar_{T_X}\varphi_t\lrcorner\nabla + \varphi_t\lrcorner\nabla(\varphi_t\lrcorner\nabla)) + \dbar_QA_t + \varphi_t\lrcorner F_{\nabla} + \varphi_t\lrcorner\nabla^QA_t + A_t\wedge A_t.$$

Note that the curvature $F_{\nabla}$ is given by
$$F_{\nabla}(\varphi,\psi) = \varphi\lrcorner\nabla(\psi\lrcorner\nabla) - (-1)^{|\varphi||\psi|}\psi\lrcorner\nabla(\varphi\lrcorner\nabla) + [\varphi,\psi]\lrcorner\nabla$$
for $\varphi,\psi\in\Omega^{0,*}(T_{\mathbb{C}}X)$. Hence
$$2\varphi_t\lrcorner\nabla(\varphi_t\lrcorner\nabla) = F_{\nabla}(\varphi_t,\varphi_t) + [\varphi_t,\varphi_t]\lrcorner\nabla.$$
However, $\varphi_t\in\Omega^{0,1}(T_X)$ and $F_{\nabla}$ is of type-$(1,1)$, we must have $F_{\nabla}(\varphi_t,\varphi_t)=0$. Therefore,
$$\varphi_t\lrcorner\nabla(\varphi_t\lrcorner\nabla) = \frac{1}{2}[\varphi_t,\varphi_t]\lrcorner\nabla.$$
As a whole we obtain
$$(\dbar_E + \varphi_t\lrcorner\nabla + A_t)^2 = \left(\dbar_{T_X}\varphi_t + \frac{1}{2}[\varphi_t,\varphi_t]\right)\lrcorner\nabla + \dbar_QA_t + \varphi_t\lrcorner F_{\nabla} + \varphi_t\lrcorner\nabla^QA_t + \frac{1}{2}[A_t,A_t].$$

But since $X_t$ is integrable, $\varphi_t$ satisfies the Maurer--Cartan equation \eqref{eqn:MC_manifold}, and so
$$(\dbar_E + \varphi_t\lrcorner\nabla + A_t)^2 = \dbar_QA_t + \varphi_t\lrcorner F_{\nabla} + \varphi_t\lrcorner\nabla^Q A_t + \frac{1}{2}[A_t,A_t].$$
Hence we conclude that $\bar{D}_t^2 = 0$ is equivalent to the following two equations
\begin{equation}\label{eqn:MC_pair_prelim}
\left\{
\begin{split}
\dbar_QA_t + \varphi_t\lrcorner F_{\nabla} + \varphi_t\lrcorner\nabla^QA_t + \frac{1}{2}[A_t,A_t] &= 0,\\
\dbar_{T_X}\varphi_t + \frac{1}{2}[\varphi_t,\varphi_t] &= 0.
\end{split}
\right.
\end{equation}

Recall that $ A(E) = Q \oplus T_X$ as smooth vector bundle. Define a bracket $[-,-]:\Omega^{0,p}(A(E))\times\Omega^{0,q}( A(E))\rightarrow\Omega^{0,p+q}(A(E))$ by
$$[(A,\varphi),(B,\psi)] := (\varphi\lrcorner\nabla^QB - (-1)^{pq}\psi\lrcorner\nabla^QA + [A,B], [\varphi,\psi]).$$
The following proposition can be proven by straightforward, but tedious, computations which we omit:
\begin{proposition}\label{prop:bracket}
The bracket $[-,-]:\Omega^{0,p}(A(E))\times\Omega^{0,q}(A(E))\rightarrow\Omega^{0,p+q}(A(E))$ defined by
$$[(A,\varphi),(B,\psi)]:=(\varphi\lrcorner\nabla^QB-(-1)^{pq}\psi\lrcorner\nabla^QA+[A,B],[\varphi,\psi])$$
satisfies
\begin{enumerate}
\item[(1)]
$[(A,\varphi),(B,\psi)]=-(-1)^{pq}[(B,\psi),(A,\varphi)],$
\item[(2)]
$[(A,\varphi),[(B,\psi),(C,\tau)]]=[[(A,\varphi),(B,\psi)],(C,\tau)]+(-1)^{pq}[(B,\psi),[(A,\varphi),(C,\tau)]],$
\end{enumerate}
for $(A,\varphi)\in\Omega^{0,p}(A(E))$, $(B,\psi)\in\Omega^{0,q}(A(E))$ and $(C,\tau)\in\Omega^{0,r}(A(E))$.
\end{proposition}

We also recall the differential operator $\dbar_{A(E)}$ defined in Section 2.

Again by direct computations, one can prove that the bracket $[-,-]$ and the Dolbeault operator $\dbar_{A(E)}$ are compatible with each other:
\begin{proposition}\label{prop:bracket_operator_compatible}
We have
$$\bar{\partial}_{A(E)}[(A,\varphi),(B,\psi)]
=[\bar{\partial}_{A(E)}(A,\varphi),(B,\psi)]+(-1)^p[(A,\varphi),\bar{\partial}_{A(E)}(B,\psi)]$$
for $(A,\varphi)\in\Omega^{0,p}(A(E))$ and $(B,\psi)\in\Omega^{0,\bullet}(A(E))$.
\end{proposition}

Propositions \ref{prop:bracket} and \ref{prop:bracket_operator_compatible} together say that $(\Omega^{0,\bullet}(A(E)),\dbar_{A(E)},[-,-])$ defines a differential graded Lie algebra (DGLA).

\begin{remark}
In the appendix, we will prove that there exists a natural isomorphism between the complex $(\Omega^{0,\bullet}(A(E)),\dbar_{A(E)})$ and the one obtained using algebraic methods \cite{Sernesi_book, Martinengo_thesis} intertwining our bracket $[-,-]$ with the algebraic one. This gives alternative proofs of Propositions \ref{prop:bracket} and \ref{prop:bracket_operator_compatible}, and shows that our DGLA is naturally isomorphic to the one derived using algebraic methods. In particular, the isomorphism class of our DGLA is independent of the choice of the hermitian metric we used to define the Chern connection $\nabla$.
\end{remark}

Using the bracket $[-,-]$ and the Dolbeault operator $\dbar_{A(E)}$, we can now rewrite the two equations \eqref{eqn:MC_pair_prelim}
as the following Maurer--Cartan equation:
\begin{align*}
\bar{\partial}_{A(E)}(A_t,\varphi_t) + \frac{1}{2}[(A_t,\varphi_t),(A_t,\varphi_t)] = 0,
\end{align*}
which governs the deformation of pairs.
We summarize our results by the following
\begin{theorem}\label{thm:MC_eqn}
Given a holomorphic pair $(X,E)$ and a smooth family of elements $\{(A_t,\varphi_t)\}_{t\in\Delta}\subset\Omega^{0,1}(A(E))$. Then $(A_t, \varphi_t)$ defines a holomorphic pair $(X_t, E_t)$ (namely, an integrable complex structure $J_t$ on $X$ together with a holomorphic bundle structure on $E$ over $(X,J_t)$)
if and only if the Maurer--Cartan equation
\begin{align}\label{eqn:MC_pair}
\bar{\partial}_{A(E)}(A_t,\varphi_t) + \frac{1}{2}[(A_t,\varphi_t),(A_t,\varphi_t)] = 0
\end{align}
is satisfied.
\end{theorem}

\section{First order deformations}\label{sec:1st_order_def}

The Maurer--Cartan equation \eqref{eqn:MC_pair} implies that a first order deformation $(A_1,\varphi_1)$ (the linear term of the Taylor series expansion of a family $(A_t,\varphi_t)$) is $\dbar_{A(E)}$-closed:
$$\bar{\partial}_{A(E)}(A_1,\varphi_1) = 0,$$
and hence defines a cohomology class in the Dolbeault cohomology group $H^{0,1}_{\dbar_{A(E)}} \cong H^1(X, A(E))$. To determine the space of first order deformations of a holomorphic pair $(X,E)$, it remains to identify isomorphic deformations.

\begin{definition}
Two deformations $\mathcal{E}\rightarrow\mathcal{X}$,
$\mathcal{E}'\rightarrow\mathcal{X}'$ of $(X,E)$ are said to be {\em isomorphic} if there exists a biholomorphism $F:\mathcal{X}\rightarrow\mathcal{X}'$ and a holomorphic bundle isomorphism $\Phi:\mathcal{E}\rightarrow\mathcal{E}'$ covering $F$ such that $F|_X = \text{id}_X$ and $\Phi|_E = \text{id}_E$.
\end{definition}

\begin{proposition}\label{prop:1st_order_deform}
Suppose $\mathcal{E}\rightarrow\mathcal{X}$ and $\mathcal{E}'\rightarrow\mathcal{X}'$ are isomorphic 1-real parameter family of deformations of $(X,E)$. If we denote by $(A_t,\varphi_t)$ and $(A'_t,\varphi'_t)$ the elements that represent the families $\mathcal{E}\rightarrow\mathcal{X}$ and $\mathcal{E}'\rightarrow\mathcal{X}'$ respectively, then there exists $(\Theta_1,v)\in\Omega^0(A(E))$ such that
$$(A'_t,\varphi'_t) = (A_t,\varphi_t) + t\dbar_{A(E)}(-\Theta_1,v) + R((A_t,\varphi_t),t(\Theta_1,v)),$$
where the error $R$ depends smoothly on $t,A(t),\varphi(t),\Theta_1,v$ and first partial derivatives of $\Theta,v$. Moreover, $R$ is of order $s^2$ in the sense that
$$R(s(A,\varphi),s(\Theta,v))=s^2R_1((A,\varphi),(\Theta,v),s),$$
for some map $R_1$ which depends smoothly (with respect to the Sobolev norm; see Section \ref{sec:completeness} for its precise definition) in $s,(A,\varphi)\in\Omega^{0,1}(A(E))$ and $(\Theta,v)\in\Omega^0(A(E))$.
\end{proposition}
\begin{proof}
As before, let $v\in\Omega^0(T_X)$ be the vector field which generates the 1-parameter family of diffeomorphisms $F_t:X\rightarrow X$ of the underlying smooth manifold $X$. Since
$$dF_t(\text{Graph}(\varphi_t:T^{0,1}_X\rightarrow T^{1,0}_X)) = \text{Graph}(\varphi'_t:T^{0,1}_X\rightarrow T^{1,0}_X),$$
we already have
$$\varphi'_t=\varphi_t+t\dbar_{T_X}v+R(\varphi_t,tv).$$
Hence it remains to show that
$$A'_t=A_t+t(\dbar_Q(-\Theta_1)-v\lrcorner F_{\nabla})+R((A_t,\varphi_t),t(\Theta_1,v)),$$
for some $\Theta_1\in\Omega^0(Q)$.

We define an endomorphism of $E$ as follows: Fix $p\in X$ and the fiber $E_p$ of $E$. Let $P_{\gamma_p(t)}:E_p\rightarrow E_{F_{t}(p)}$ be the parallel transport along $t\longmapsto\gamma_p(t):=F_{t}(p)$. Define $\Theta_t:=P_{\gamma_p(t)}^{-1}\Phi_{t}:E_p\rightarrow E_p$. Then $\Theta_t$ defines a bundle endomorphism of $E$.
Let us write
$$\Theta_t=\Theta_0+t\Theta_1+O(t^2),$$
$$A'_t=A'_0+tA'_1+O(t^2),$$
Since $\Phi_{0}=\text{id}_E$, we have $\Theta_0=\text{id}_E$ and $A'_0=0$. We need to compute $A'_1$.

Now let $e_t$ be a local holomorphic section of $E_t\subset\mathcal{E}$. Since $\Phi_{t}$ is holomorphic, $\Phi_{t}(e_t)$ is a holomorphic section of $E'_t\subset\mathcal{E}'$, so that $\bar{D}'_t\Phi_{t}(e_t)=0$, i.e.
$$\bar{\partial}_E\Phi_{t}(e_t)=-\varphi'(t)\lrcorner\nabla\Phi_{t}(e_t)-A'_t\Phi_{t}(e_t).$$
We want to compute the first derivatives of both sides of this equation with respective to $t$ at $t=0$.

First note that $\Phi_{t}=P_{\gamma(t)}\Theta_t$, so we have
$$\pd{}{t}\Phi_{t}(e_t)|_{t=0}=-v\lrcorner\nabla e+\Theta_1(e)+\pd{}{t}e_t|_{t=0},$$
where $e=e_0$. Hence
\begin{align*}
\pd{}{t}\bar{\partial}_E\Phi_{t}(e_t)|_{t=0}
& = \dbar_E\left(\pd{}{t}\Phi_{t}(e_t)|_{t=0}\right)\\
& = -\dbar_E(v\lrcorner\nabla e)+\dbar_E(\Theta_1(e))+\dbar_E\left(\pd{}{t}e_t|_{t=0}\right).
\end{align*}

For the term, $-\varphi'_t\lrcorner\nabla\Phi_{t}(e_t)$, we have
\begin{align*}
\pd{}{t}(-\varphi'_t\lrcorner\nabla\Phi_{t}(e_t))|_{t=0}
& = -\varphi'_1\lrcorner\nabla e-\dbar_{T_X}v\lrcorner\nabla e\\
& = -\varphi'_1\lrcorner\nabla e-\dbar_E(v\lrcorner\nabla e)-v\lrcorner\dbar_E\nabla e.
\end{align*}
Since $e=e_0$ is holomorphic with respective to $\mathcal{E}|_{\pi^{-1}(0)}$, we have
$$\dbar_E\nabla e=F_{\nabla}(e).$$
Moreover, since $e_t$ is holomorphic with respective to $E_t$, we have $\bar{D}_te_t=0$, that is,
$$\dbar_Ee_t=-\varphi_t\lrcorner\nabla e_t-A_te_t.$$
Differentiate with respective to $t$ and set $t=0$, we obtain
$$-\varphi_1\lrcorner\nabla e=\dbar_E\left(\pd{}{t}e(t)|_{t=0}\right)+A_1e.$$
Hence
$$\pd{}{t}(-\varphi'_t\lrcorner\nabla\Phi_{t}(e_t))|_{t=0}=\dbar_E\left(\pd{}{t}e_t|_{t=0}\right)-\dbar_E(v\lrcorner\nabla e)-v\lrcorner F_{\nabla}e+A_1e.$$

For the term $-A'_t\Phi_{t}(e_t)$, we have
$$\pd{}{t}(-A'_t\Phi_{t}(e_t))|_{t=0}=-A'_1e.$$

As a whole we obtain the formula
$$\dbar_E(\Theta_1(e))=-v\lrcorner F_{\nabla}(e)+(A_1-A'_1)(e).$$
Since $e$ is holomorphic with respective to $\mathcal{E}|_{\pi^{-1}(0)}$, $\dbar_E(\Theta_1(e))=(\dbar_Q\Theta_1)(e)$, so that
$$A'_1=A_1+\dbar_Q(-\Theta_1)-v\lrcorner F_{\nabla}.$$
Therefore we have
$$\pd{}{t}(A'_t-A_t)|_{t=0}=A'_1-A_1=\dbar_Q(-\Theta_1)-v\lrcorner F_{\nabla},$$
or in other words,
$$A'_t=A_t+t(\dbar_Q(-\Theta_1)-v\lrcorner F_{\nabla})+O(t^2).$$

Since $A'_t$ is completely determined by $(A_t,\varphi_t)$ and $t(\Theta_1,v)$, we have
$$A'_t=A_t+t(\dbar_Q(-\Theta_1)-v\lrcorner F_{\nabla})+R((A_t,\varphi_t),t(\Theta_1,v)),$$
where $R$ is of order $t^2$ and depends smoothly on $t,A_t,\varphi_t,\Theta_1,v$. Moreover, since the equation
$$\bar{\partial}_E\Phi_{t}(e_t)=-\varphi'_t\lrcorner\nabla\Phi_{t}(e_t)-A'_t\Phi_{t}(e_t)$$
depends smoothly on first order partial derivatives of $\Theta_1$ and $v$, we see that the error $R$ also depends smoothly on first order partial derivatives of $\Theta_1$ and $v$.

Finally, $R$ satisfies
$$R(s(A,\varphi),st(\Theta,v))=s^2R_1((A,\varphi),(\Theta_1,v),s),$$
for some map $R_1$ which depends smoothly in $s,(A,\varphi)\in\Omega^{0,1}(A(E))$ and $(\Theta,v)\in\Omega^0(A(E))$. This follows from the fact that
\begin{align*}
R((A,\varphi),(\Theta,v)) & \to 0 \quad \text{as $(\Theta,v)\to 0$, and}\\
R((A,\varphi),(\Theta,v)) & \to R((\Theta,v)) \quad \text{as $(A,\varphi)\to 0$},
\end{align*}
with $R(s(\Theta,v))=s^2R((\Theta,v))$.
\end{proof}

\begin{corollary}
If $\mathcal{E}\rightarrow\mathcal{X}$ and $\mathcal{E}'\rightarrow\mathcal{X}'$ are isomorphic deformations of $(X,E)$, then the first order terms $(A_1,\varphi_1)$ and $(A'_1,\varphi'_1)$ of the corresponding families $(A_t,\varphi_t)$ and $(A'_t,\varphi'_t)$ respectively differ by an $\dbar_{A(E)}$-exact form.
\end{corollary}
\begin{proof}
We have $A'_1-A_1 = \dbar_Q(-\Theta_1)-v\lrcorner F_{\nabla}$ and $\varphi'_1-\varphi_1 = \dbar_{T_X}v$, whence $$(A'_1,\varphi'_1)-(A_1,\varphi_1)=\dbar_{A(E)}(-\Theta_1,v).$$
\end{proof}

We conclude that the space of first order deformations of a holomorphic pair $(X,E)$ is precisely given by the Dolbeault cohomology group $H^{0,1}_{\dbar_{A(E)}} \cong H^1(X, A(E))$

\section{Obstructions and Kuranishi family}\label{sec:obs_kuranishi}

Now given a first order deformation $[(A_1,\varphi_1)] \in H^{0,1}_{\dbar_{A(E)}} \cong H^1(X, A(E))$, it is standard in deformation theory to ask whether one can find a family $(A_t,\varphi_t)$ integrating $(A_1,\varphi_1)$ to give an actual family of deformations. To study this problem, we use a method due to Kuranishi \cite{Kuranishi65}.

We need to review several operators commonly used in Hodge theory. We first choose a hermitian metric $g$ on $X$ and $h$ on $A(E)$, so that we can define a hermitian product $(\cdot,\cdot)$ on $\Omega^{0,\bullet}( A(E))$. Define the formal adjoint of $\dbar_{ A(E)}$ with respective to $(\cdot,\cdot)$ by
$$\left(\dbar_{A(E)}\alpha,\beta\right) = \left(\alpha,\dbar_{A(E)}^*\beta\right).$$
Then the {\em Laplacian} is defined by
$$\Delta_{ A(E)}:=\dbar_{ A(E)}\dbar_{ A(E)}^*+\dbar_{ A(E)}^*\dbar_{ A(E)}.$$
This is an elliptic self-adjoint operator and thus has a finite dimensional kernel $\mathbb{H}^p(X, A(E))$, consisting of {\em harmonic forms}. We have the standard isomorphism from Hodge theory:
$$H^p(X, A(E)) \cong H^{0,p}_{\dbar_{ A(E)}} \cong \mathbb{H}^p(X, A(E)).$$

Take a completion of $\Omega^{0,\bullet}( A(E))$ with respective to $(\cdot,\cdot)$ to get a Hilbert space $L^*$, and let $H:L^*\rightarrow\mathbb{H}^*(X, A(E))$ be the harmonic projection. The Green's operator $G:L^*\rightarrow L^*$ is defined by
$$I=H+\Delta_{ A(E)}G=H+G\Delta_{ A(E)}.$$
It commutes with $\dbar_{ A(E)}$ and $\dbar_{ A(E)}^*$.

Now let $\eta_1,\dots,\eta_m\in \mathbb{H}^1(X, A(E))$ be a basis and $\epsilon_1(t):=\sum_{j=1}^{m}t_j\eta_j\in\mathbb{H}^1(X, A(E))$. Consider the equation
$$\epsilon(t)=\epsilon_1(t)-\frac{1}{2}\dbar_{ A(E)}^*G[\epsilon(t),\epsilon(t)].$$
We denote the H\"older norm by $\norm{\cdot}_{k,\alpha}$. The following estimates are obvious:
\begin{align*}
\norm{\dbar_{ A(E)}^*\epsilon}_{k,\alpha} \leq & C_1\norm{\epsilon}_{k+1,\alpha}\\
\norm{[\epsilon,\delta]}_{k,\alpha} \leq & C_2\norm{\epsilon}_{k+1,\alpha}\norm{\delta}_{k+1,\alpha}.
\end{align*}
In \cite{Douglis-Nirenberg55}, Douglis and Nirenberg proved the following nontrivial a priori estimate:
$$\norm{\epsilon}_{k,\alpha} \leq C_3(\norm{\Delta_{ A(E)}\epsilon}_{k-2,\alpha}+\norm{\epsilon}_{0,\alpha}).$$
Applying these and following the proof of \cite[Chapter 4, Proposition 2.3]{Morrow-Kodaira_book}, one can deduce an estimate for the Green's operator $G$:
$$\norm{G\epsilon}_{k,\alpha}\leq C_4\norm{\epsilon}_{k-2,\alpha},$$
where all $C_i$'s are positive constants which depend only on $k$ and $\alpha$.

Then by the same argument as in \cite[Chapter 4, Proposition 2.4]{Morrow-Kodaira_book}, or alternatively using an implicit function theorem for Banach spaces \cite{Kuranishi65}, we obtain a unique solution $\epsilon(t)$ which satisfies the equation
$$\epsilon(t)=\epsilon_1(t)-\frac{1}{2}\dbar_{ A(E)}^*G[\epsilon(t),\epsilon(t)],$$
and is analytic in the variable $t$.
Note that the solution $\epsilon(t)$ is always smooth. Indeed, by applying the Laplacian to the above equation, we get
$$\Delta_{ A(E)}\epsilon(t)+\frac{1}{2}\dbar_{ A(E)}^*[\epsilon(t),\epsilon(t)]=0.$$
Also, the solution $\epsilon(t)$ is holomorphic in $t$, so we have
$$\dpsum{j}{}\secpd{\epsilon(t)}{t_j}{\bar{t}_j}=0.$$
Now since the operator
$$\Delta_{ A(E)}+\dpsum{j}{}\secpd{}{t_j}{\bar{t}_j}$$
is elliptic, we see that $\epsilon(t)$ is smooth by elliptic regularity.

Following Kuranishi \cite{Kuranishi65} (see also \cite[Chapter 4]{Morrow-Kodaira_book}), we have the following
\begin{proposition}
The solution $\epsilon(t)$ that satisfies
$$\epsilon(t)=\epsilon_1(t)-\frac{1}{2}\dbar_{ A(E)}^*G[\epsilon(t),\epsilon(t)]$$
solves the Maurer--Cartan equation if and only if $H[\epsilon(t),\epsilon(t)]=0$, where $H$ is the harmonic projection.
\end{proposition}
\begin{proof}
Suppose the Maurer--Cartan equation holds. Then
$$H[\epsilon(t),\epsilon(t)]=2H\dbar_{ A(E)}\epsilon(t)=0.$$

Conversely, suppose that $H[\epsilon(t),\epsilon(t)]=0$. We must show that
$$\delta(t):=\dbar_{ A(E)}\epsilon(t)+\frac{1}{2}[\epsilon(t),\epsilon(t)]=0.$$
Recall that $\epsilon(t)$ is a solution to
$$\epsilon(t)=\epsilon_1(t)-\frac{1}{2}\dbar_{ A(E)}^*G[\epsilon(t),\epsilon(t)]$$
and $\epsilon_1(t)$ is $\dbar_{ A(E)}$-closed. By applying $\dbar_{ A(E)}$ to this equation, we get
$$\dbar_{ A(E)}\epsilon(t)=-\frac{1}{2}\dbar_{ A(E)}\dbar_{ A(E)}^*G[\epsilon(t),\epsilon(t)].$$
Hence
$$2\delta(t)=\dbar_{ A(E)}\dbar_{ A(E)}^*G[\epsilon(t),\epsilon(t)]-[\epsilon(t),\epsilon(t)].$$
Using the Hodge decomposition on forms, we can write
$$[\epsilon(t),\epsilon(t)]=H[\epsilon(t),\epsilon(t)]+\Delta_{ A(E)}G[\epsilon(t),\epsilon(t)]
=\Delta_{ A(E)}G[\epsilon(t),\epsilon(t)].$$
Therefore, we have
\begin{align*}
2\delta(t)
= (\Delta_{ A(E)}G-\dbar_{ A(E)}\dbar_{ A(E)}^*G)[\epsilon(t),\epsilon(t)]
= \dbar_{ A(E)}^*\dbar_{ A(E)}G[\epsilon(t),\epsilon(t)]
= 2\dbar_{ A(E)}^*G[\dbar_{ A(E)}\epsilon(t),\epsilon(t)],
\end{align*}
and hence
\begin{align*}
\delta(t)
= \dbar_{ A(E)}^*G[\dbar_{ A(E)}\epsilon(t),\epsilon(t)]
= \dbar_{ A(E)}^*G[\delta(t)-\frac{1}{2}[\epsilon(t),\epsilon(t)],\epsilon(t)]
= \dbar_{ A(E)}^*G[\delta(t),\epsilon(t)],
\end{align*}
where we have used the Jacobi identity in the last equality. Using the estimate
$$\norm{[\xi,\eta]}_{k,\alpha}\leq C_{k,\alpha}\norm{\xi}_{k+1,\alpha}\norm{\eta}_{k+1,\alpha},$$
we get
$$\norm{\delta(t)}_{k,\alpha}\leq C_{k,\alpha}\norm{\delta(t)}_{k,\alpha}\norm{\epsilon(t)}_{k,\alpha}.$$
By choosing $|t|$ to be small enough so that $C_{k,\alpha}\norm{\epsilon(t)}_{k,\alpha}<1$, we must have $\delta(t)=0$ for all $|t|$ small enough. This finishes the proof.
\end{proof}

In the case when $H[\epsilon(t),\epsilon(t)]$ vanishes identically (which always holds if $H^2(X, A(E))=0$), we have the following
\begin{corollary}
If $H[\epsilon(t),\epsilon(t)]=0$ for all $t$, then we have a complex analytic family $\mathcal{E}\rightarrow\mathcal{X}$.
\end{corollary}
\begin{proof}
If $H[\epsilon(t),\epsilon(t)]=0$ for all $t$, then $\epsilon(t)=(A_t,\varphi_t)$ satisfies the Maurer--Cartan equation and so $(X_t,E_t)$ is holomorphic for each $t$. In particular, we obtain a deformation $\mathcal{X}$ of $X$. Let $\mathcal{E}:=\Delta\times E$. A smooth section $\sigma:\mathcal{X}\rightarrow\mathcal{E}$ of $\mathcal{E}$ on $\mathcal{X}$ can be written as
$$\sigma:(t,x)\longmapsto (t,s(x,t)),$$
for some smooth map $s:\mathcal{X}\rightarrow E$. We define a Dolbeault operator $\dbar_{\mathcal{E}}:\Omega^0_{\mathcal{X}}(\mathcal{E})\rightarrow\Omega^{0,1}_{\mathcal{X}}(\mathcal{E})$ on $\mathcal{E}$ by
$$\dbar_{\mathcal{E}}\sigma(t,x)=(t,\dbar_{E_t}s(t,x)).$$
Note that $\dbar_{\mathcal{E}}$ is well-defined for, if $\{e_k(t,x)\}$ are local holomorphic frame of $E_t\rightarrow X_t$, then we can write
$$\dbar_{\mathcal{E}}\sigma(t,x)=(t,\dbar_{\mathcal{E}}(s^k(t,x)e_k(t,x))=\dbar_ts^k(t,x)\otimes e_k(t,x)),$$
which is a smooth section of $\Omega^{0,1}_{\mathcal{X}}(\mathcal{E})$. Clearly, $\dbar_{\mathcal{E}}^2=0$ and hence $\mathcal{E}$ is a holomorphic vector bundle over $\mathcal{X}$.
\end{proof}

In general, the condition $H^2(X, A(E))=0$ may not be satisfied. But we can define the (singular) analytic space
$$\mathcal{S}:=\{t\in\Delta:H[\epsilon(t),\epsilon(t)]=0\}$$
and form a family $\mathcal{E}\rightarrow\mathcal{X}$ over $\mathcal{S}$, which is called the {\em Kuranishi family} of $(X,E)$.
In particular, we see that the obstruction space is precisely given by the Dolbeault cohomology group $H^{0,2}_{\dbar_{ A(E)}} \cong H^2(X, A(E))$, and the obstructions to deformations of a holomorphic pair $(X,E)$ is captured by the Kuranishi map
\begin{align*}
Ob_{(X,E)}:U\subset H^1(X, A(E))\to H^2(X, A(E)),\ \dpsum{i=1}{m}t_j\eta_j\mapsto H[\epsilon(t),\epsilon(t)],
\end{align*}
where $U$ is a small open subset around the origin $0\in H^1(X,A(E))$ whose diameter is less than twice of the radius of convergence of $\epsilon(t)$.

\section{A proof of completeness}\label{sec:completeness}

The goal of this section is to give a proof of the local completeness of a Kuranishi family for the deformation of the pair $(X,E)$. Existence of a locally complete (or versal) family for deformations of pairs was first proved by Siu-Trautmann \cite{Siu-Trautmann81}. Here we give another proof using Kuranishi's method.

\begin{definition}
A family $\mathcal{E}\rightarrow\mathcal{X}$ over a analytic space $\mathcal{S}$ is said to be {\em locally complete (or versal)} if for any family $\mathcal{E}'\rightarrow\mathcal{X}'$ over a sufficiently small ball $\Delta$, there exists a analytic map $f:\Delta\rightarrow\mathcal{S}$ such that the family $\mathcal{E}'\rightarrow\mathcal{X}'$ is the pull-back of $\mathcal{E}\rightarrow\mathcal{X}$ via $f$.
\end{definition}

Recall that for given $\epsilon_1(t)\in\mathbb{H}^1(X, A(E))$, we have existence of solutions $\epsilon(t)$ to
$$\epsilon(t)=\epsilon_1(t)-\frac{1}{2}\dbar_{ A(E)}^*G[\epsilon(t),\epsilon(t)]$$
and $\epsilon(t)$ satisfies the Maurer--Cartan equation if and only if $H[\epsilon(t),\epsilon(t)]=0$. We then obtain an analytic family $\mathcal{E}\rightarrow\mathcal{X}$ over
$$\mathcal{S}:=\{t\in\Delta:H[\epsilon(t),\epsilon(t)]=0\}.$$
The main theorem is as follows:
\begin{theorem}\label{thm:completeness}
The Kuranishi family $\mathcal{E}\rightarrow\mathcal{X}$ over $\mathcal{S}$ is locally complete.
\end{theorem}

Before going into the details of the proof, we first introduce the Sobolev norm:
One can endow $ A(E)$ a hermitian metric $H$, induced from that of $E$ and $X$, and define the inner product
$$(\alpha,\beta)_k:=\dpsum{|I|\leq k}{}\int_XH(D^I\alpha,D^I\beta),$$
$\alpha,\beta\in\Omega^{0,*}( A(E))$. The Sobolev norm is defined by
$$|\alpha|:=(\alpha,\alpha)^{\frac{1}{2}}.$$
One has the estimate
$$|[\alpha,\beta]|_k\leq C_k|\alpha|_{k+1}|\beta|_{k+1},$$
for some constant $C_k>0$.

We take a completion of $\Omega^{0,\bullet}( A(E))$ with respective to $(\cdot,\cdot)_k$ to get a Hilbert space $L_k^*$. The harmonic projection $H:L_k^{\bullet}\rightarrow\mathbb{H}^{\bullet}(X, A(E))$ and the Green's operator $G:L_k^{\bullet}\rightarrow L_{k+2}^{\bullet}$ satisfy the estimates
$$|H\alpha|_k\leq C_k|\alpha|_k,$$
$$|\dbar_{ A(E)}^*G\alpha|_k\leq C_k|\alpha|_{k-1}.$$
The following lemma will be useful in the proof of the completeness theorem.
\begin{lemma}\label{lem:small_sol}
For fixed $\epsilon_1(t)\in\mathbb{H}^1(X, A(E))$, $t\in\mathcal{S}$, the equation
$$\epsilon(t)=\epsilon_1(t)-\frac{1}{2}\dbar_{ A(E)}^*G[\epsilon(t),\epsilon(t)]$$
has only one small solution.
\end{lemma}
\begin{proof}
Suppose $\epsilon$ is another solution. Let $\delta:=\epsilon-\epsilon(t)$. Then
\begin{align*}
\delta
& = -\frac{1}{2}\dbar_{ A(E)}^*G([\epsilon,\epsilon]-[\epsilon(t),\epsilon(t)])\\
& = -\frac{1}{2}\dbar_{ A(E)}^*G([\delta,\epsilon(t)]+[\epsilon(t),\delta]+[\delta,\delta])\\
& = -\frac{1}{2}\dbar_{ A(E)}^*G(2[\delta,\epsilon(t)]+[\delta,\delta]).
\end{align*}
Hence
$$|\delta|_k\leq C_k(|\delta|_k|\epsilon(t)|_k+|\delta|_k^2)\leq C_k|\delta|_k(|\epsilon(t)|_k+|\delta|_k).$$
For $|\epsilon(t)|_k$ and $|\epsilon|_k$ small, we can only have $|\delta|_k=0$.
\end{proof}

We are now ready to prove the local completeness theorem.
\begin{proof}[Proof of Theorem \ref{thm:completeness}]
Let $\mathcal{E}'\rightarrow\mathcal{X}'$ be a deformation of $(X,E)$. Let $\epsilon'$, be the element representing this deformation. We first prove that if $\dbar_{ A(E)}^*\epsilon'=0$, then there exists $t\in\mathcal{S}$ such that $\epsilon'=\epsilon(t)$.

Note that $\epsilon'$ satisfies the Maurer--Cartan equation:
$$\dbar_{ A(E)}\epsilon'+\frac{1}{2}[\epsilon',\epsilon']=0.$$
Applying $\dbar_{ A(E)}^*$, we get
$$\dbar_{ A(E)}^*\dbar_{ A(E)}\epsilon'+\frac{1}{2}\dbar_{ A(E)}^*[\epsilon',\epsilon']=0.$$
Since $\dbar_{ A(E)}^*\epsilon'=0$, we have
$$\Delta_{ A(E)}\epsilon'+\frac{1}{2}\dbar_{ A(E)}^*[\epsilon',\epsilon']=0.$$
Then using $I=H+G\Delta_{ A(E)}$, we get
$$\epsilon'=H\epsilon'-\frac{1}{2}\dbar_{ A(E)}^*G[\epsilon',\epsilon'].$$
Note that $H\epsilon'\in\mathbb{H}^1(X, A(E))$ and by the estimate $|H\epsilon'|_k\leq C_k|\epsilon'|$, we see that $|H\epsilon'|_k$ is small if $|\epsilon'|_k$ is small. Hence $H\epsilon'=\epsilon_1(t)$ for some $t\in\mathcal{S}$. Therefore, if the ball $\Delta$ is small enough, $\epsilon'$ is a solution to
$$\epsilon'=\epsilon_1(t)-\frac{1}{2}\dbar_{ A(E)}^*G[\epsilon',\epsilon'].$$
Therefore, $\epsilon'=\epsilon(t)$ for some $t\in\mathcal{S}$ by Lemma \ref{lem:small_sol}.

Now we prove that for any given small deformation $\mathcal{E}'\rightarrow\mathcal{X}'$, one can find an isomorphic deformation $\mathcal{E}''\rightarrow\mathcal{X}''$ such that the element $\epsilon''$ which represents the family $\mathcal{E}''\rightarrow\mathcal{X}''$ is $\dbar_{A(E)}^*$-closed. This will prove the local completeness.
Indeed, we will prove the following: Given a deformation $\epsilon'$, there exists $\eta\in Im(\dbar_{ A(E)}^*)\subset\Omega^0( A(E))$ such that the element $\epsilon''$, which represents the deformation $\mathcal{E}''\rightarrow\mathcal{X}''$, is $\dbar_{ A(E)}^*$-closed.

Let $\eta=(\Theta,v)\in\Omega^0( A(E))$, then the elements $\epsilon',\epsilon''$, which represent the deformations $\mathcal{E}'\rightarrow\mathcal{X}'$ and $\mathcal{E}''\rightarrow\mathcal{X}''$ respectively, satisfy
$$\epsilon''=\epsilon'+\dbar_{ A(E)}\eta+R(\epsilon',\eta),$$
where the error term $R$ satisfies $R(s\epsilon',s\eta)=s^2R_1(\epsilon',\eta,s)$ as in Proposition \ref{prop:1st_order_deform}. Hence $\dbar_{ A(E)}^*\epsilon''=0$ if and only if
$$\dbar_{ A(E)}^*\epsilon'+\dbar_{ A(E)}^*\dbar_{ A(E)}\eta+\dbar_{ A(E)}^*R(\epsilon',\eta)=0.$$
If $\eta\in Im(\dbar_{ A(E)}^*)$, then
$$\Delta_{ A(E)}(\eta)+\dbar_{ A(E)}^*R(\epsilon'(s),\eta)+\dbar_{ A(E)}^*\epsilon'=0.$$
Applying $G$, we get
$$\eta+\dbar_{ A(E)}^*GR(\epsilon',\eta)+\dbar_{ A(E)}^*G\epsilon'=0.$$
Let $U_1\subset L_k^1$ and $V_1\subset L_k^0$ be neighborhoods around $0$ such that $R(\epsilon',\eta)$ is defined. Define $F:U_1\times V_1\rightarrow L_k^0$ by
$$F(\epsilon',\eta):=\eta+\dbar_{ A(E)}^*GR(\epsilon',\eta)+\dbar_{ A(E)}^*G\epsilon'.$$
By the order condition on the error term $R$, the derivative of $F$ with respective to $\eta$ at $(0,0)$ is the identity map. Hence by the implicit function theorem, there is a $C^{\infty}$ function $g$ such that $F(\epsilon',\eta)=0$ if and only if $\eta=g(\epsilon')$. By the error condition again, the (second order) operator $|\dbar_{ A(E)}^*R(\epsilon',-)|_k$ is small if $|\epsilon'|_k$ is small. Hence
$$\Delta_{ A(E)}+\dbar_{ A(E)}^*R(\epsilon',-)+\dbar_{ A(E)}^*\epsilon'$$
is still a quasi-linear elliptic operator. By elliptic regularity, $\eta$ is smooth. This completes our proof.
\end{proof}

\section{Unobstructed deformations}\label{sec:unobstr}

In this section, we investigate various circumstances under which deformations of holomorphic pairs are unobstructed.
%We will also apply our results to show that the dimension of the cohomology group $H^1(X,\text{End}(T_X))$ is invariant under small deformations of an algebraic K3 surface, answering a question of Huybrechts \cite{Huybrechts95} in the 2-dimensional case.

%\subsection{Unobstructed deformations of pairs}

To begin with, note that we have an exact sequence of holomorphic vector bundles
$$0\longrightarrow \text{End}(E)\longrightarrow  A(E)\longrightarrow T_X\longrightarrow 0$$
by the construction of $ A(E)$ (which shows that $ A(E)$ is an extension of $Q = \text{End}(E)$ by $T_X$).
This induces a long exact sequence in cohomology groups:
\begin{align*}
\cdots & \longrightarrow H^1(X,Q)\longrightarrow H^1(X, A(E))\longrightarrow H^1(X,T_X)\longrightarrow\\
       & \longrightarrow H^2(X,Q)\longrightarrow H^2(X, A(E))\longrightarrow H^2(X,T_X)\longrightarrow\cdots,
\end{align*}
and the first order term $(A_1,\varphi_1)$ defines a class $[(A_1,\varphi_1)]\in H^1(X, A(E))$.

The following proposition, which first appeared in \cite[Appendix]{Huybrechts95} without proof, describes the relations between the deformations of a pair $(X,E)$ and that of $X$ and $E$.
\begin{proposition}\label{prop:obstr_commute}
Denote the Kuranishi obstruction maps of the deformation theories of $X$, $E$ and $(X,E)$ by $Ob_X$, $Ob_E$ and $Ob_{(X,E)}$ respectively. Then, wherever the obstruction maps are defined, we have the following commutative diagram:
\begin{equation*}
\xymatrix{
\cdots\ar@{->}[r]& H^1(X,Q) \ar[d]_{Ob_E} \ar@{->}[r]^{\iota^*}& {H^1(X, A(E))} \ar[d]_{Ob_{(X,E)}} \ar@{->}[r]^{\pi^*}& H^1(X,T_X) \ar[d]_{Ob_X} \ar@{->}[r]^{\delta} & \cdots
\\\cdots\ar@{->}[r] & {H^2(X,Q)} \ar@{->}[r]^{\iota^*} & {H^2(X, A(E))} \ar@{->}[r]^{\pi^*} & H^2(X,T_X) \ar@{->}[r]^{\delta} & \cdots
}
\end{equation*}
Here, the connecting homomorphism $\delta$ is given by contracting with the Atiyah class:
$$\delta(\varphi)=\varphi\lrcorner[F_{\nabla}].$$
\end{proposition}
\begin{proof}
By definition,
$$\iota^*([A])=[(A,0)],\quad \pi^*([(A,\varphi)])=[\varphi].$$
The commutative diagram follows directly from the definitions of the maps $Ob_X$,$Ob_E$ and $Ob_{(X,E)}$.
\end{proof}

\begin{remark}
We remark that since $tr[A,A]=0$ for any $A\in\Omega^{0,1}(Q)$, the obstruction of deforming $E$ (with $X$ fixed) actually lies in $H^2(X,\text{End}_0(E))$, where $\text{End}_0(E)\subset\text{End}(E)$ is the trace-free part of $\text{End}(E)$.
\end{remark}

\begin{remark}
For any $[(A,\varphi)]\in H^1(X, A(E))$ such that $Ob_{(X,E)}(A,\varphi) = 0$, we have
$$0=Ob_X\circ\pi^*[(A,\varphi)]=Ob_X([\varphi]).$$
In this case, the map $(A, \varphi) \mapsto \varphi$ induces a map of {\em Kuranishi slices}, i.e. every deformation of the pair $(X,E)$ induces a deformation of the manifold $X$.
\end{remark}

An immediate consequence of this proposition is the following slight generalization of a result in \cite{Pan13}:
\begin{proposition}\label{prop:unobstr_basic}
Suppose that $Ob_X\circ\pi^*=0$ and the connecting homomorphism $\delta:H^1(X,T_X)\rightarrow H^2(X,Q)$ is surjective, then deformations of the pair $(X,E)$ are unobstructed.
\end{proposition}
\begin{proof}
Surjectivity of $\delta$ implies that the map $\iota^*:H^2(X,Q)\rightarrow H^2(X, A(E))$ is a zero map, and hence the map $\pi^*:H^2(X, A(E))\rightarrow H^2(X,T_X)$ is injective. But $\pi^*\circ Ob_{(X,E)}=Ob_X\circ\pi^*=0$, so we have $Ob_{(X,E)}=0$.
\end{proof}

In the case when $E=L$ is a line bundle, we recover the following
\begin{corollary}[\cite{Pan13}, Lemma 2.4]
Let $X$ be a compact complex manifold with unobstructed deformations and $L$ be a holomorphic line bundle over $X$ such that the map
$$\cup c_1(L):H^1(X,T_X)\rightarrow H^2(X,\mathcal{O}_X)$$
is surjective. Then deformations of the pair $(X,L)$ are unobstructed.
\end{corollary}

For example, if $X$ is an $n$-dimensional compact K\"ahler manifold with trivial canonical line bundle, then $X$ admits unobstructed deformations. If we further assume that $H^{0,2}(X) = 0$ (e.g. when the holonomy of $X$ is precisely $SU(n)$), then deformations of $(X,L)$ for any line bundle $L$ are unobstructed.

\begin{definition}
A holomorphic vector bundle $E$ over a compact complex manifold $X$ is said to be {\em good} if $H^2(X,Q_0)=0$, where $Q_0$ is the trace-free part of $Q = \text{End}(E)$.
\end{definition}

\begin{proposition}\label{prop:unobstr_K3}
Let $X$ be a compact complex surface with trivial canonical line bundle (e.g. a K3-surface), and let $E$ be a good bundle over $X$ with $c_1(E)\neq 0$. Then deformations of the pair $(X,E)$ are unobstructed.
\end{proposition}
\begin{proof}
By the theorem of Tian and Todorov \cite{Tian87, Todorov89}, we have $Ob_X=0$. Hence the condition $Ob_X\circ\pi^*=0$ is automatic.

On the other hand, note that $Q^* = (E^* \otimes E)^* = E^* \otimes E = Q$ and similarly $Q_0^* = Q_0$.
By Serre duality and the fact that $K_X\cong\mathcal{O}_X$, we have $H^0(X, Q_0) \cong (H^2(X, Q_0))^* = 0$ since $E$ is good.
This implies that $H^0(X, Q) \cong H^0(X, \mathcal{O}_X) \cong \bC$.
Then applying Serre duality again gives
$$H^2(X, Q) \cong (H^0(X, Q^*\otimes K_X))^* \cong (H^0(X, Q))^* \cong \bC.$$
In this case, the connecting homomorphism $\delta:H^1(X,T_X)\cong H^{1,1}(X)\rightarrow H^2(X,Q)\cong\mathbb{C}$ is simply given by
$$\delta(\varphi)=\int_X\varphi\cup [trF_{\nabla}]=-2\pi i\int_X\varphi\cup c_1(E).$$
When $c_1(E)\neq 0$, $\delta$ is a nonzero map and hence surjective. Proposition \ref{prop:unobstr_basic} then says that any deformation of $(X,E)$ is unobstructed.
\end{proof}

\appendix

\section{Comparison with the algebraic approach}\label{sec:compare_classical}

The aim of this appendix is to give an explicit comparison between the analytic approach we adopt here and the classical algebraic approach (see the book \cite{Sernesi_book} for the deformation theory of $(X,L)$ where $L$ is a holomorphic line bundle on $X$, and the thesis \cite{Martinengo_thesis} for the general case).

We start with a definition
\begin{definition}
A {\em differential operator of order 1} on a vector bundle $E$ is a linear map $P:\Omega^0(E)\rightarrow\Omega^0(E)$ such that locally,
$$P=(g_{ij})+\dpsum{k}{}h_{ij}^{k}\pd{}{x^{k}},$$
with $(g_{ij})$ be a matrix with entries in $\mathcal{O}_X(U_{\alpha})$ and $h_{ij}^k\in\mathcal{O}_X(U_{\alpha})$.

A differential operator of order 1 is said to be {\em with scalar principle symbol} if $h_{ij}^{k}=h^{k}\cdot I$.
\end{definition}
In the algebraic approach, the role of the Atiyah extension $ A(E)$ is replaced by the sheaf $D^1(E)$ of scalar differential operators of order less than $1$ on $E$, namely, we have an exact sequence
$$0\longrightarrow \text{End}(E)\longrightarrow D^1(E)\longrightarrow T_X\longrightarrow 0,$$
where the surjective map $\sigma:D^1(E)\rightarrow T_X$ is locally defined by the symbol
$$\sigma(P)=\dpsum{k}{}h^{k}\pd{}{x^k}.$$
There is an obvious identification of $D^1(E)$ with $ A(E)$ as smooth vector bundles, but we will see that in fact $D^1(E)$ can be given a holomorphic structure so that $D^1(E)$ and $ A(E)$ are isomorphic as holomorphic vector bundles.

First of all, locally on an open set $U_{\alpha}$, we can write
$$P|_{U_{\alpha}}=g_{\alpha}+d_{\alpha}.$$
Let $e_{\alpha}$ be local sections of $E$, $\{f_{\alpha\beta}\}$ be holomorphic transition functions of $E$ and $P_{\alpha}:=P|_{U_{\alpha}}(e_{\alpha})$. To get a global differential operator, we must have
$$f_{\alpha\beta}P_{\beta}=P_{\alpha}f_{\alpha\beta}.$$
Hence
$$g_{\beta}=f_{\beta\alpha}g_{\alpha}f_{\alpha\beta}+f_{\beta\alpha}(d_{\alpha}f_{\beta\alpha}),\quad d_{\alpha}=d_{\beta}.$$
Set
$$\tau_{\alpha}:=g_{\alpha}-d_{\alpha}\lrcorner (\bar{h}^{-1}_{\alpha}\partial\bar{h}_{\alpha}),$$
where $h_{\alpha}$ is the Hermitian metric on $E|_{U_{\alpha}}$. Define a map
$$\Phi:g_{\alpha}+d_{\alpha}\longmapsto (\tau_{\alpha},d_{\alpha}).$$
Straightforward computations give the identities
$$f_{\alpha\beta}\tau_{\beta}=\tau_{\alpha}f_{\alpha\beta},\quad \dbar_{ A(E)}(\tau_{\alpha},d_{\alpha})=0.$$
It then follows that $\Phi$ defines an isomorphism between $D^1(E)$ and $ A(E)$. So we can give $D^1(E)$ a holomorphic structure by pulling back that on $ A(E)$ via $\Phi$. Hence we obtain

\begin{proposition}
$D^1(E)$ carries a natural holomorphic structure so that it is isomorphic to the Atiyah extension $ A(E)$. In particular,
$H^p(X,D^1(E))\cong H^p(X, A(E))$ for any $p$.
\end{proposition}

Together with the Lie bracket \cite{Martinengo_thesis}
$$[\omega\otimes P,\eta\otimes Q]:=\omega\wedge\eta\otimes[P,Q]+\omega\wedge\lie{\sigma(P)}{\eta}\otimes Q-(-1)^{|\omega||\eta|}\eta\wedge\lie{\sigma(Q)}{\omega}\otimes P,$$
the triple $(\Omega^{0,*}(D^1(E)),\dbar,[-,-])$ forms a DGLA. Note that the Lie derivative acts by
$$\lie{X}{\omega}=d(i_X\omega)+i_Xd\omega=i_X\partial\omega,$$
for any $\omega\in\Omega^{0,*}(X)$ and $X\in\Omega^0(T_X)$.

%We prove that $[-,-]$ is same as the bracket $[-,-]_h$ (which depends on the metric $h$ on $E$) we have defined before.
\begin{theorem}
The isomorphism $\Phi:D^1(E)\rightarrow A(E)$
$$\Phi:g_{\alpha}+d_{\alpha}\longmapsto g_{\alpha}-\bar{h}_{\alpha}^{-1}d_{\alpha}\bar{h}_{\alpha}$$
intertwines with the brackets $[-,-]$ and $[-,-]_h$, i.e.
$$\Phi[\varphi\otimes P,\psi\otimes Q]=[\varphi\otimes\Phi(P),\psi\otimes\Phi(Q)]_h.$$
\end{theorem}
\begin{proof}
We first prove that
$$\Phi[P,Q]=[\Phi(P),\Phi(Q)]_h.$$
Write $P=g+d$ and $Q=g'+d'$ locally in a coordinate neighborhood $U\subset X$. Then
$$[P,Q]=[g,g']+dg'-d'g+[d,d']$$
and so
$$\Phi[P,Q]=([g,g']+dg'-d'g-\bar{h}^{-1}[d,d']\bar{h},[d,d']).$$
On the other hand,
$$[\Phi(P),\Phi(Q)]_h
=(\nabla^Q_d(g'-\bar{h}^{-1}d'\bar{h})-\nabla^Q_{d'}(g-\bar{h}^{-1}d\bar{h})+[g-\bar{h}^{-1}d'\bar{h},g-\bar{h}^{-1}d'\bar{h}],[d,d'])$$
Now, we compute
\begin{align*}
  & \nabla^Q_d(g'-\bar{h}^{-1}d'\bar{h})-\nabla^Q_{d'}(g-\bar{h}^{-1}d\bar{h})\\
= & d(g'-\bar{h}^{-1}d'\bar{h})+[\bar{h}^{-1}d\bar{h},g'-\bar{h}^{-1}d'\bar{h}]-d'(g+\bar{h}^{-1}d\bar{h})-[\bar{h}^{-1}d'\bar{h},g-\bar{h}^{-1}d\bar{h}]\\
= & dg'-d'g+[\bar{h}^{-1}d\bar{h},g']-[\bar{h}^{-1}d'\bar{h},g]-d\bar{h}^{-1}d'\bar{h}+d'\bar{h}^{-1}d\bar{h}-2[\bar{h}^{-1}d\bar{h},\bar{h}^{-1}d'\bar{h}].
\end{align*}
and
$$[g-\bar{h}^{-1}d\bar{h},g'-\bar{h}^{-1}d'\bar{h}]
=[g,g']-[g,\bar{h}^{-1}d'\bar{h}]-[\bar{h}^{-1}d\bar{h},g']+[\bar{h}^{-1}d\bar{h},\bar{h}^{-1}d'\bar{h}].$$
Therefore, their sum equals to
$$[g,g']+dg'-d'g-d\bar{h}^{-1}d'\bar{h}+d'\bar{h}^{-1}d\bar{h}-[\bar{h}^{-1}d\bar{h},\bar{h}^{-1}d'\bar{h}].$$
Finally,
\begin{align*}
[\bar{h}^{-1}d\bar{h},\bar{h}^{-1}d'\bar{h}]
& = \bar{h}^{-1}(d\bar{h})\bar{h}^{-1}(d'\bar{h})-\bar{h}^{-1}(d'\bar{h})\bar{h}^{-1}(d\bar{h})\\
& = -(d\bar{h}^{-1})(d'\bar{h})+(d'\bar{h}^{-1})(d\bar{h})\\
& = -d(\bar{h}^{-1}d'\bar{h})+\bar{h}^{-1}dd'\bar{h}+d'(\bar{h}^{-1}d\bar{h})-\bar{h}^{-1}d'd\bar{h}\\
& = -d(\bar{h}^{-1}d'\bar{h})+d'(\bar{h}^{-1}d\bar{h})+\bar{h}^{-1}[d,d']\bar{h}.
\end{align*}
Hence
\begin{align*}
  & \nabla^Q_d(g'-\bar{h}^{-1}d'\bar{h})-\nabla^Q_{d'}(g-\bar{h}^{-1}d\bar{h})+[g-\bar{h}^{-1}d'\bar{h},g-\bar{h}^{-1}d'\bar{h}]\\
= & [g,g']+dg'-d'g-\bar{h}^{-1}[d,d']\bar{h},
\end{align*}
which is the required equality.

To prove the general case, we have, by linearity and the case that we have proved, the $\text{End}(E)$-part of $\Phi[\omega\otimes P,\eta\otimes Q]$ is equal to
\begin{align*}
  & \omega\wedge\eta\otimes[\tau(P),\tau(Q)]_h-\omega\wedge\eta\otimes[\tau(P),\sigma(Q)]_h+\omega\wedge\eta\otimes[\sigma(P),\tau(Q)]_h\\
+ & \omega\wedge\lie{\sigma(P)}{\eta}\otimes\tau(Q)-(-1)^{|\omega||\eta|}\eta\wedge\lie{\sigma(Q)}{\omega}\otimes\tau(P),
\end{align*}
where $\tau(P):=pr_{\text{End}(E)}\circ\Phi(P)$. On the other hand, the $\text{End}(E)$-part of $[\Phi(\omega\otimes P),\Phi(\eta\otimes Q)]_h$ is equal to
$$((\omega\otimes\sigma(P))\lrcorner\nabla^Q(\eta\otimes\tau(Q))-(-1)^{|\omega||\eta|}(\eta\otimes\sigma(Q))\lrcorner\nabla^Q(\omega\otimes\tau(P))
+\omega\wedge\eta\otimes[\tau(P),\tau(Q)]_h.$$
The Leibniz rule for connections implies that
\begin{align*}
((\omega\otimes\sigma(P))\lrcorner\nabla^Q(\eta\otimes\tau(Q))
& = \omega\wedge\lie{\sigma(P)}{\eta}\otimes\tau(Q) + \omega\wedge\eta\otimes\nabla^Q_{\sigma(P)}\tau(Q)\\
& = \omega\wedge\lie{\sigma(P)}{\eta}\otimes\tau(Q) + \omega\wedge\eta\otimes[\sigma(P),\tau(Q)].
\end{align*}
Similarly, we have
$$((\eta\otimes\sigma(Q))\lrcorner\nabla^Q(\omega\otimes\tau(P))
=\eta\wedge\lie{\sigma(Q)}{\omega}\otimes\tau(P)+(-1)^{|\omega||\tau|}\omega\wedge\eta\otimes[\sigma(Q),\tau(P)].$$
Putting these back into $[\Phi(\omega\otimes P),\Phi(\eta\otimes Q)]_h$, we get
$$\Phi[\omega\otimes P,\eta\otimes Q]=[\Phi(\omega\otimes P),\Phi(\eta\otimes Q)]_h.$$
This proves our theorem.
\end{proof}

\begin{remark}
This theorem gives a proof of the required identities in Propositions \ref{prop:bracket} and \ref{prop:bracket_operator_compatible}, and the fact that the DGLA $(\Omega^{0,\bullet}( A(E)),\dbar_{ A(E)},[-,-]_h)$ is independent of the choice of the hermitian metric $h$.
\end{remark}

\bibliographystyle{amsplain}
\bibliography{geometry}

\end{document}